\documentclass[11pt]{amsart}
\usepackage{amsmath,amssymb,amsthm}
\usepackage[dvipdfmx]{graphicx}
\usepackage[usenames]{color}
\usepackage[shortlabels]{enumitem}
\hypersetup{% hyperrefオプションリスト
setpagesize=false,
 bookmarksnumbered=true,%
 bookmarksopen=true,%
 colorlinks=true,%
 linkcolor=blue,
 citecolor=red,
}

\newtheorem{thm}{Theorem}[section]
\newtheorem{lem}[thm]{Lemma}
\newtheorem{cor}[thm]{Corollary}

\theoremstyle{definition}

\newtheorem{problem}[thm]{Problem}

\newcommand{\CPb}{\overline{\mathbb{CP}}{}^{2}}
\newcommand{\CP}{{\mathbb{CP}}{}^{2}}

\begin{document}

\title[Knot surgered elliptic surfaces for a $(2,2h+1)$-torus knot]
{Knot surgered elliptic surfaces without 1- and 3-handles for a $(2,2h+1)$-torus knot}
%%% \subtitle{Sub-Title of Your Article}%%% optional

\author[N. Monden]{Naoyuki Monden}
\address{Department of Mathematics, Faculty of Science, \\Okayama University, 
Okayama 700-8530, Japan}
\email{n-monden@okayama-u.ac.jp}

\author{Reo Yabuguchi}
\address{Department of Mathematics, Faculty of Science, \\Okayama University, 
Okayama 700-8530, Japan}
\email{p21o36b2@s.okayama-u.ac.jp}

%\subjclass[2020]{Primary 57K41; Secondary 57K43}% Subject code(s)

%\keywords{Lefschetz fibration, Geography problem, fiber sum indecomposable}% Key word(s)

\begin{abstract}
We show that for any positive integer $h$, a knot surgered elliptic surface $E(n)_{T(2,2h+1)}$ for a $(2,2h+1)$-torus knot $T(2,2h+1)$ and the elliptic surface $E(1)_{2,2h+1}$ admit handle decompositions without 1- and 3-handles using the Kirby diagrams %``on surfaces" 
derived from Lefschetz fibrations on them. 
\end{abstract}

\maketitle

\section{Introduction}\label{Intro}
It is a central problem in $4$-dimensional differential topology to determine whether a given smooth $4$-manifold admits an exotic smooth structure or not. 
For example, between 2005 and 2010, various exotic smooth structures on $\CP \sharp n\CPb$ $(2\leq n\leq 8)$ had been intensively constructed by many authors (for example, \cite{Park2, SS, FS6, PSS, AP1, AP2}). 
However, it remains open whether $\mathbb{S}^4$ and $\CP$ admit an exotic smooth structure or not. 
If such a structure exists, then any handle decomposition of it must contain at least either a 1- or 3-handle (cf. \cite{Yasui1}). 
In contrast, there is the following, which is known as Problem 4.18 in Kirby's problem list \cite{Kirby}:
\begin{problem}\label{problem:1}
Does every simply connected, closed 4-manifold have a handlebody decomposition without 1-handles? Without 1- and 3-handles?
\end{problem} 
In fact, many classical simply connected 4-manifolds are known to admit neither 1- nor 3-handles in their handle decompositions (see, for example, \cite{GS, Yasui2, Akbulut2}). 
Note that if a 4-manifold admits a handle decomposition without 1-handles or 3-handles, then it is simply connected. 
In this paper, we focus on knot surgered elliptic surfaces for a $(2,2h+1)$-torus knot for Problem~\ref{problem:1}.

%To state our main result, we explain a knot surgery operation introduced by Fintushel-Stern \cite{FS3}. 
%Let $\nu(K)$ be an open tubular neighborhood of a knot $K$ in $\mathbb{S}^3$. 
%Suppose that a 4-manifold $X$ contains an embedded torus $T$ of self-intersection $0$, and hence a (closed) tubular neighborhood of $T$ in $X$ can be identified with $T \times \mathbb{D}^2$, where $\mathbb{D}^2$ is the 2-dimensional disk. 
%The \textit{knot surgery manifold} $X_K$ is defined by $X_K = (X- \mathrm{Int}(T \times \mathbb{D}^2)) \cup (\mathbb{S}^1 \times (\mathbb{S}^3 - \nu(K)))$, where they are glued together in such a way that the homology class $[pt\times \partial \mathbb{D}^2]$ is identified with $[pt\times \lambda]$, and $\lambda$ is the class of a longitude of $K$. 
%Note that if $X-T$ is simply connected, then $X_K$ is homeomorphic to $X$. 
For a 4-manifold $X$ and a knot $K$ in $\mathbb{S}^3$, we denote by $X_K$ the \textit{knot surgered manifold} defined by Fintushel-Stern \cite{FS3}. 
Let $E(n)$ be the simply connected elliptic surface, and let $T_{p,q}$ be a $(p,q)$-torus knot for relatively prime integers $p$ and $q$. 
%Now, we state our main result. 
Our main result is the following. 
\begin{thm}\label{thm:1}
For any positive integer $h$ and $n\geq 1$, $E(n)_{T(2,2h+1)}$ admits a handle decomposition without 1- and 3-handles. 
\end{thm}

In this paper, we show Theorem~\ref{thm:1} by using a Kirby diagram derived from a Lefschetz fibration on $E(n)_{T(2,2h+1)}$. %by introducing a Kirby diagram ``on a surface" derived from a Lefschetz fibration on $E(n)_{T(2,2h+1)}$. 
We give some related results. 
Okamori \cite{Okamori} constructed an exotic $\CP\sharp 9\CPb$ without $1$-handles in a handle decomposition using a standard Kirby diagram derived from its genus-$2$ Lefschetz fibration structure. 
Based on the idea, the second author \cite{Yabuguchi1, Yabuguchi2} gave a handle decomposition of $E(1)_K$ without $1$-handles for any genus-$1$ fibered knot $K$. 
Baykur \cite{Baykur} showed that there exist infinitely many irreducible 4-manifolds, which are obtained by fiber summing of Lefschetz fibrations given in \cite{BH}, with prescribed signature and spin type admitting handle decompositions without $1$- and $3$-handles using their Lefschetz fibration structures (and without using explicit Kirby diagrams). 
As a corollary, he gave a negative answer to Problem 4.91 in Kirby's problem list \cite{Kirby}.

Here, we present some topics on Problem~\ref{problem:1} for $E(n)_{p,q}$ ($n\geq 1$, $p,q\geq 2$ and $\gcd(p,q)=1$), which is defined a complex surface obtained by a pair of $p$- and $q$-log transforms on $E(n)$. 
Harer, Kas and Kirby \cite{HKK} conjectured that every handle decomposition of $E(1)_{2,3}$, which is an exotic $\CP \sharp 9\CPb$, must have both 1- and 3-handles. 
Moreover, the following is noted by Gompf in \cite{Gompf}: \textit{it is a good conjecture that $E(n)_{p,q}$ has no handle decomposition without 1- and 3-handles.} 
Around 2010, these conjectures were disproved in \cite{Yasui2, Akbulut2}. 
Akbulut \cite{Akbulut2} proved that $E(1)_{2,3}$ admits a handle decomposition without $1$- and $3$-handles, and Yasui \cite{Yasui2} showed that $E(n)_{p,q}$ has a handle decomposition without $1$-handles for $n\geq 1$ and $(p, q) = (2, 3), (2, 5), (3, 4), (4, 5)$. 
In \cite{Akbulut1}, Akbulut showed that for $n=1,2,\ldots$, there is an infinite family $X_n$ of mutually non-diffeomorphic exotic copies of $E(1)$ without $1$- and $3$-handles such that $X_1 = E(1)_{2,3}$ and $X_n = E(1)_{K_n}$, where $K_n$ is are knots with distinct Alexander polynomials. 
See also \cite{Yasui1,Yasui3,Okamori} for related results.

Recently, their alternative proofs or generalizations on \cite{Yasui2, Akbulut2} were given in \cite{Sakamoto, Yabuguchi1, Yabuguchi2, Kusuda, Tange}. 
Sakamoto in \cite{Sakamoto} (resp. the second author \cite{Yabuguchi1,Yabuguchi2}) constructed a handle decomposition of $E(1)_{2,7}$ (resp. $E(1)_{2,3}$) without 1-handles. 
Kusuda \cite{Kusuda} showed that $E(n)_{5,6}$, $E(n)_{6,7}$, $E(n)_{7,8}$ and $E(n)_{8,9}$ have handle decompositions without $1$-handles for $n\geq 4$, $n\geq 5$, $n\geq 9$ and $n\geq 24$, respectively. 
Taki \cite{Taki} studied an upper bound of the minimal number $l$ that $E(n)_{p,q} \sharp l\CPb$ has a handle decomposition without 1-handles.
Tange \cite{Tange} showed that $E(n)_K$ admits a handle decomposition with no $1$-handles for a knot $K$ with bridge number at most $9n$, and hence, $E(1)_{p,q}$ also admits a handle decomposition with no $1$-handles for $\min\{p,q\}\leq 9$ since $E(1)_{T_{p,q}}$ is diffeomorphic to $E(1)_{p,q}$ (see \cite{FS5, Park}). 
He also constructed a handle decomposition of $E(n)_{p,q}$ without $1$-handles for $\min\{p,q\} \leq 4$. 
Note that $E(n)_{p,q}$ is not diffeomorphic to $E(n)_{T(p,q)}$ for $n > 1$ since their Seiberg-Witten invariants are different. 
%We do not know that the handle decompositions given in \cite{Sakamoto, Kusuda, Tange} have no 1- and 3-handles. 
These results are about only 1-handles, whereas 
%On the other hand, 
we obtain the following consequence from Theorem~\ref{thm:1}. 
\begin{cor}\label{cor:1}
For any positive integer $h$, $E(1)_{2,2h+1}$ admits a handle decomposition without 1- and 3-handles. 
\end{cor}

The outline of the paper is as follows.
In Section~\ref{LF}, we present the basics of Lefschetz fibrations. 
In Section~\ref{Kirbydiagram}, we briefly review Kirby diagrams derived from Lefschetz fibrations and introduce Kirby diagrams ``on surfaces". 
Section~\ref{LFonE(n)_K} gives facts on Lefschetz fibrations on $E(n)_K$ for a fibered knot $K$. 
In the last section, we prove Theorem~\ref{thm:1}.

\section{Basics of Lefschetz fibrations}\label{LF}
In this section, we present some definitions and facts concerning Lefschetz fibrations. 
More details can be found in \cite{Matsumoto, GS}.

Let $X$ be a closed, connected, oriented, smooth $4$-manifold, and let $B$ be the $2$-dimensional disk $\mathbb{D}^2$ or the $2$-dimensional sphere $\mathbb{S}^2$. 
We say that a proper smooth map $f: X \to B$ is a genus-$g$ \textit{Lefschetz fibration} if a regular fiber of $f$ is diffeomorphic to a Riemann surface $\Sigma_g$ of genus $g$, all critical values of $f$ lie in $\mathrm{Int} B$ and for each critical point $p$ and its image $f(p)$, there are complex local coordinate charts agreeing with the orientations of $X$ and $B$ for which $f$ is of the form $f(z_1,z_2) = z_1z_2$. 
Throughout this paper, we require that the restriction of $f$ to the set $C$ of critical points is injective and that no fiber contains a sphere of self-intersection $-1$. 
For a genus-$g$ Lefschetz fibration, any fiber containing a critical point is called a \textit{singular fiber}, which is obtained by collapsing a simple closed curve, called the \textit{vanishing cycle}, in a nearby regular fiber to the critical point.

Let $\Gamma_g^b$ be \textit{mapping class group} of the compact oriented surface $\Sigma_g^b$ obtained by removing $b$ disjoint open disks from $\Sigma_g$, that is the group of isotopy classes of orientation preserving self-diffeomorphisms of $\Sigma_g^b$. 
We assume that diffeomorphisms and isotopies fix the points of the boundary. 
To simplify notation, we write $\Sigma_g = \Sigma_g^0$ and $\Gamma_g = \Gamma_g^0$. 
In this paper, we use the same symbol for a diffeomorphism and its isotopy class, or a simple closed curve and its isotopy class. 
For $\phi_1$ and $\phi_2$ in $\Gamma_g^b$, the notation $\phi_2\phi_1$ means that we first apply $\phi_1$ and then $\phi_2$. 
Let $t_c$ be the Dehn twist about a simple closed curve $c$ on $\Sigma_g^b$. 
It is well-known that the relation $t_{\phi(c)}=\phi t_c\phi^{-1}$ holds for an element $\phi$ in $\Gamma_g^b$. 
For a product $V=t_{v_m} \cdots t_{v_2} t_{v_1}$ of Dehn twist, we set ${}_{\phi}(V) = t_{\phi(v_m)} \cdots t_{\phi(v_2)} t_{\phi(v_1)}$ for an element $\phi$ in $\Gamma_g^b$.

The \textit{global monodromy} of a genus-$g$ Lefschetz fibration $f:X \to \mathbb{D}^2$ with the vanishing cycles $v_1,\ldots,v_n$ of the singular fibers comprises a product of right-handed Dehn twist in $\Gamma_g$ as $V=t_{v_m} \cdots t_{v_2} t_{v_1} \in \Gamma_g$. 
Conversely, we obtain a genus-$g$ Lefschetz fibration over $\mathbb{D}^2$ with the vanishing cycles $v_1,\ldots,v_n$ from $V=t_{v_m} \cdots t_{v_2} t_{v_1}$ in $\Gamma_g$. 
These hold for a genus-$g$ Lefschetz fibration over $\mathbb{S}^2$, but we require that $V = \mathrm{id}$.

Two Lefschetz fibrations $f_1:X_1\to B$ and $f_2:X_2\to B$ are said to be \textit{isomorphic} if there exist orientation preserving diffeomorphisms $H: X_1\to X_2$ and $h: B \to B$ such that $f_2\circ H = h\circ f_1$.  
It is well-known that when we apply a cyclic permutation to the global monodromy $V=t_{v_m} \cdots t_{v_2} t_{v_1}$ of a genus-$g$ Lefschetz fibration $f: X \to B$, the genus-$g$ Lefschetz fibration $f'$ with the resulting global monodromy is the same as the original one, and therefore $f'$ is isomorphic to $f$. 
Moreover, by theorems of Kas \cite{Kas} and Matsumoto \cite{Matsumoto}, a Lefschetz fibration $f'$ with a global monodromy ${}_{\phi}(V)$ for any element $\phi$ in $\Gamma_g$ is isomorphic to $f$.

\section{Kirby diagrams}\label{Kirbydiagram}
In this section, we introduce a Kirby diagram on a surface with a boundary component derived from a Lefschetz fibration.  

\subsection{A Kirby diagram derived from a Lefschetz fibration}\label{KirbydiagramLF}
We recall how to draw a Kirby diagram of $X$ admitting a genus-$g$ Lefschetz fibration $f: X \to \mathbb{D}^2$ with a global monodromy $t_{v_m} \cdots t_{v_2} t_{v_1}$ in $\Gamma_g$. 
For the terminology, we refer to the reader to \cite{GS}.

We start with a handle decomposition of $\Sigma_g \times \mathbb{D}^2$ with one 0-handle, $2g$ 1-handles, and one $0$ framed 2-handle, which runs over all 1-handles, from a fixed handle decomposition of $\Sigma_g$. 
For example, when we consider a handle decomposition of $\Sigma_g$ as in the upper part of Figure~\ref{HDandKD1}, the Kirby diagram of $\Sigma_g \times \mathbb{D}^2$ obtained from the handle decomposition of $\Sigma_g$ is depicted in the lower part of Figure~\ref{HDandKD1}. 
Thus, we get a handle decomposition of $X$ by attaching $m$ $2$-handles $h_1,h_2,\ldots,h_m$ to $\Sigma_g \times \mathbb{D}^2$ along the simple closed curves $v_1,v_2,\ldots,v_m$ on different fibers of $\Sigma_g \times \mathbb{S}^1 \to \mathbb{S}^1$ with framing on less than the product framing of $\Sigma_g \times \mathbb{S}^1$. 
From the observation above, a Kirby diagram of $X$ is obtained from the Kirby diagram of $\Sigma_g \times \mathbb{D}^2$ by adding $m$ $2$-handles $h_1,h_2,\ldots,h_m$ along vanishing cycles $v_1,v_2,\ldots,v_m$ with framing $-1$ in parallel level with index increasing toward the reader. 
The framing of the 2-handle derived from $\Sigma_g \times \mathbb{D}^2$ is $0$. 
(For example, the right part of Figure~\ref{HDandKD2} depicts a Kirby diagram of $X_0$ admitting the genus-$2$ Lefschetz fibration $f_0: X_0 \to \mathbb{D}^2$ with the global monodromy $t_{x_3} t_{x_2} t_{x_1}$, where $x_1,x_2,x_3$ are the simple closed curves on $\Sigma_2$ as in the left part of Figure~\ref{HDandKD2}.)

In this paper, we employ the dotted circle notation for $1$-handles. 
We fix the subsurface $\Sigma_g^1$ $(\subset \mathbb{S}^3)$ obtained by removing the $2$-dimensional 2-handle from the fixed handle decomposition of $\Sigma_g$. 
Then, we redraw this Kirby diagram of $X$ ``on $\Sigma_g^1$" so that $(-1)$-framed 2-handles are drawn on $\mathrm{Int} \Sigma_g^1(\subset \mathbb{S}^3)$ preserving over- and under-crossing information, the $0$-framed attaching circle is the boundary curve of $\Sigma_g^1$, and each ($4$-dimensional) 1-handle encircles the belt sphere of each $2$-dimensional 1-handle of $\Sigma_g^1$ (see, for example, the left part of Figure~\ref{HDandKD3}). 
Strictly speaking, the Kirby diagram in $\mathbb{S}^3$ is regarded with that in $\Sigma_{g}^1 \times [0,1] (\subset \mathbb{S}^3)$, but for convenience, we consider all 2-handles in the Kirby diagram as those on $\Sigma_g^1$. 
Note that by forgetting the surface $\Sigma_g^1$, we obtain a standard Kirby diagram of $X$ (see, for example, the right part of Figure~\ref{HDandKD3}). 
Let $\alpha_1,\beta_1,\alpha_2,\beta_2,\ldots,\alpha_g,\beta_g$ be the standard generators of $\pi_1(\Sigma_g^1,p)$ at the base point $p$ so that $\alpha_i$ (resp. $\beta_i$) passes through $(2i-1)$-th (resp. $2i$-th) $2$-dimensional $2$-handle of $\Sigma_g^1$ only once and does not pass through the other handles (see for example, Figure~\ref{1handles}). 
We denote by $\alpha_i^\ast$ (resp. $\beta_i^\ast$) the $4$-dimensional 1-handle (or the dotted circle) encircling $\alpha_i$ (resp. $\beta_i$).

\subsection{A Kirby diagram on a surface with a boundary component}\label{KirbydiagramS}
It is troublesome to draw simple closed curves or 2-handles on the fixed subsurface $\Sigma_g^1$ arising from the fixed handle decomposition of $\Sigma_g$. 
For this reason, we first draw simple closed curves or 2-handles on $\Sigma_g$ (preserving over- and under-crossing information), and take a small open disk $D$ on $\Sigma_g$. 
Moreover, we fix an ambient isotopy which deforms $\Sigma_g - D$ to $\Sigma_g^1$. 
Next, we consider handle silding operations ``on $\Sigma_g - D$". 
Finally, by deforming $\Sigma_g - D$ to $\Sigma_g^1$ by the isotopy, we obtain a Kirby diagram ``on $\Sigma_g^1$". 
Here, we omit the framings of 2-handles since we don't need it to prove Theorem~\ref{thm:1}. 
Each small empty box in Figures means some full twist, but we also omit the number since we don't need it to prove Theorem~\ref{thm:1}. 
For example, let $h_1,h_2$ be the $2$-handles on $\Sigma_2-D$ as in the left upper part of Figure~\ref{diagram3}, where $h_1'$ is a framing of $h_1$. 
We slide $h_2$ over $h_1$ as in the left lower part of this figure. 
The right part of this Figure expresses the corresponding operation for the 2-handles $h_1,h_2$ on the $\Sigma_2^1$ arising from the given handle decomposition of $\Sigma_2$. 
Then, by isotopy, we obtain $h_1$ and a new $2$-handle as in Figure~\ref{diagram4}.

The dotted circles (i.e. 4-dimensional 1-handles) in a Kirby diagram on $\Sigma_g-D$ can be drawn using the fixed ambient isotopy from the dotted circles (i.e. 4-dimensional 1-handles) in a given Kirby diagram on $\Sigma_g^1$. 
However, it is also troublesome to draw them explicitly (cf. the left and right parts of Figure~\ref{deform1handle}. The thin dashed lines go through the inside of $\Sigma_g-D$). 
For this reason, we consider dotted circles as dotted arcs as follows: 
We project dotted circles (i.e., 1-handles) $\alpha_i^\ast, \beta_i^\ast$ to belt spheres of $2$-dimensional 1-handles of $\Sigma_g^1$, and the images are drawn as ``dotted segments'' so that all 2-handles in $\int \Sigma_g^1$ either are disjoint from them or undercross them. 
By deforming $\Sigma_g^1$ to $\Sigma_g-D$ by the fixed ambient isotopy, we regard $2g$ $1$-handles to ``$2g$ disjoint arcs with dots connecting two disjoint points in $\partial \Sigma_g - D$'' (see, for example, Figure~\ref{cancelingoperation0}). 
The dotted arcs are ``dual arcs" of the standard generators of $\pi_1(\Sigma_g-D,p)$, and the disk $D$ is a ``dual disk" of $p$.

Figure~\ref{cancelingoperation1} (with the framing indicated by double-strand notation) shows the effect of a handle canceling operation of a canceling 1-handle/2-handle pair in a Kirby diagram on $\Sigma_g^1$. 
This is the same as that in the standard Kirby diagram corresponding to the diagram on $\Sigma_g^1$ since this is obtained by forgetting $\Sigma_g^1$. 
%However, when we apply a handle canceling operation to a Kirby diagram on $\Sigma_g - D$, we need to pay attention as follows: 
%Let us consider a canceling pair of a $1$-handle $H$ and a $2$-handle $h$.  
%When there are other 2-handles $h_1,h_2,\ldots,h_m$ running over $H$, we slide $h_i$ to $h$ so that we form a bund-sum of $h$ and $h_i$ along a band parallel to the dotted arc corresponding to $H$ (see, for example, Figure~\ref{cancelingoperation0}, with the framing indicated by double-strand notation). 
%Then, we apply canceling operation to the pair $(H,h)$. 
In this paper, we only apply a handle canceling operation on $\Sigma_g - D$ (see, for example, Figure~\ref{cancelingoperation0}, with the framing indicated by double-strand notation), which corresponds to it inside of the $0$-framed 2-handle derived from $\Sigma_g \times \mathbb{D}^2$ as in the lower part of Figure~\ref{HDandKD1}, 
although it is no problem to apply it on $D$, which corresponds to it outside of the $0$-framed 2-handle.

A Kirby diagram on a surface (i.e. on $\Sigma_g - D$) has the advantage that we can directly draw 2-handles from vanishing cycles. 
It can be found in the proof of Theorem~\ref{thm:1} in Section~\ref{proof:thm1}.

\section{Knot surgered elliptic surfaces}\label{LFonE(n)_K}
Fintsuhel-Stern \cite{FS4} showed that for a genus-$h$ fibered knot $K$, $E(n)_K$ admits a genus-$(2h+n-1)$ Lefschetz fibration over $\mathbb{S}^2$. 
In particular, the fibration is a fiber sum of two copies of a genus-$(2h+n-1)$ Lefschetz fibration $f_{h,n}$ on $(\Sigma_h \times \mathbb{S}^2) \sharp 4\CPb$ by gluing a diffeomorphism $\Phi_K$ on $\Sigma_{2h+n-1}$. 
We omit the definition of a fiber sum operation, but we explain the precise definition of $\Phi_K$. 
Let us decompose $\Sigma_{2h+n-1}$ into three surfaces $\Sigma_h^1$, $\Sigma_{n-1}^2$ and $\Sigma_{h}^1$, which are the subsurfaces of the right, middle and left of $\Sigma_{2h+n-1}$ as in Figure~\ref{gurtas}, respectively. 
The diffeomorphism $\Phi_K$ is the map so that the restriction of $\Phi_K$ to the first $\Sigma_h^1$ is $\phi_K$, and the restrictions to $\Sigma_{n-1}^2$ and the second $\Sigma_h^1$ is the identity maps, where $\phi_K$ is the monodromy of the genus-$h$ fibered knot $K$. 
Yun \cite{Yun2} gave an explicit global monodromy of the Lefschetz fibration on $E(n)_K$. 
To state this, let us introduce some notations and the background on the global monodromy. 
Let $c_1,c_2,\ldots,c_{2n-1},D_0,D_1,\ldots,D_{2h}$ be the simple closed curves on $\Sigma_{2h+n-1}$ as in Figure~\ref{gurtas}, we set 
\begin{align*}
W = t_{c_{2n-2}} \cdots t_{c_2} t_{c_1} t_{c_1} t_{c_2} \cdots t_{c_{2n-2}} t_{D_0} t_{D_1} \cdots t_{D_{2h}} t_{c_{2n-1}}. 
\end{align*}
Gurtas \cite{Gurtas} showed that the isotopy class of a certain involution $\iota$ on $\Sigma_{2h+n-1}$ is expressed as $W$, and hence $W^2=\mathrm{id}$ in $\Gamma_{2h+n-1}$ (see also \cite{Matsumoto, C, Korkmaz}). 
Yun \cite{Yun1} also verified this fact, up to Hurwitz equivalence, and showed that the global monodromy of the Lefschetz fibration $f_{h,n}$ is $W$. 
We now can state the global monodromy of the Lefschetz fibration on $E(n)_K$. 
\begin{thm}[\cite{FS4},\cite{Yun2}]\label{monodromyEnK}
Let $K$ be a genus-$h$ fibered knot in $\mathbb{S}^3$. 
Then the knot surgered 4-manifold $E(n)_K$ admits a genus-$(2h+n-1)$ Lefschetz fibration $f: E(n)_K \to \mathbb{S}^2$ with the global monodromy ${}_{\Phi_K}(W)^2 \cdot W^2$. 
\end{thm}

Note that the monodromy $\phi_{T_{2,2h+1}}$ in $\Gamma_{h}^1$ of a $(2,2h+1)$-torus knot $T_{2,2h+1}$ is 
\begin{align*}
\Phi_{T_{2,2h+1}} = t_{a_{2h}}^{-1} \cdots t_{a_{2}}^{-1} t_{a_{1}}^{-1},  
\end{align*}
where $a_1,a_2,\ldots,a_{2h}$ are the simple closed curves on $\Sigma_h^1$ as a subsurface of $\Sigma_{2h+n-1}$ as in Figure~\ref{gurtas} (see, for example \cite{Yun2}).

\section{Proof of Theorem~\ref{thm:1}}\label{proof:thm1}
In this section, we prove Theorem~\ref{thm:1}. 
The following two lemmas are used to show Theorem~\ref{thm:1}. 
\begin{lem}\label{lem:1000}
Let $f:X \to \mathbb{S}^2$ be a genus-$g$ Lefschetz fibration, and let $f_1: X_1\to \mathbb{D}_1$ and $f_2: X_2 \to \mathbb{D}_2$ be genus-$g$ Lefschetz fibrations by cutting the base $\mathbb{S}^2$ of $f$ into two disks $\mathbb{D}_1$ and $\mathbb{D}_2$, respectively. 
If $X_1$ and $X_2$ admit handle decompositions without 1- and 3-handles, respectively, then $X$ admits a handle decomposition without 1- and 3-handles.   
\end{lem}
\begin{proof}
%We note that a Lefschetz fibration over the disk has a handle decomposition without 3- and 4-handles (see Section~\ref{Kirbydiagram}). 
%Therefore, f
From the assumption, the handle decompositions of $X_1$ and $X_2$ only have one 0-handle and some 2-handles. 
By turning the handlebody of $X_2$ ``upside down'', we obtain a handle decomposition of $X_2$ which only has some 2-handles and one 4-handle. 
Since $X$ is obtained by gluing $X_1$ and $X_2$ along their boundaries, we get a required handle decomposition of $X$. 
\end{proof}

\begin{lem}\label{lem:1001}
Let $a_1,a_2,\ldots,a_{2h},c_1,c_2,\ldots,c_{2n-1},D_0,D_1,\ldots,D_{2h}$ be simple closed curves on $\Sigma_{2h+n-1}$ as in Figure~\ref{gurtas}, and set $D_{2h}' = t_{a_1}t_{a_2} \cdots t_{a_{2h}}(D_{2h})$. 
The handle decomposition of $X$ arising from the Lefschetz fibration $f:X \to \mathbb{D}^2$ with the global monodromy $W \cdot W'$ has no 1-handles, where 
\begin{align*}
W &= t_{c_{2n-2}} \cdots t_{c_2} t_{c_1} t_{c_1} t_{c_2} \cdots t_{c_{2n-2}} t_{D_0} t_{D_1} \cdots t_{D_{2h}} t_{c_{2n-1}}, \\
W' &= t_{c_{2n-2}} \cdots t_{c_2} t_{c_1} t_{c_1} t_{c_2} \cdots t_{c_{2n-2}} t_{t_{a_1}(D_0)} t_{t_{a_2}(D_1)} \cdots t_{t_{a_{2h}}(D_{2h-1})} t_{D'_{2h}} t_{c_{2n-1}}, 
\end{align*}
\end{lem}
We give a proof of Lemma~\ref{lem:1001} after the proof of Theorem~\ref{thm:1}.

\begin{proof}[Proof of Theorem~\ref{thm:1}]
To shorten notation, we write $K$ instead of a $(2,2h+1)$-torus knot $T_{2,2h+1}$. 
Recall the genus-$(2h+n-1)$ Lefschetz fibration $f : E(n)_{K} \to \mathbb{S}^2$ with the global monodromy ${}_{\Phi_K}(W)^2 \cdot W^2$ defined in Section~\ref{LFonE(n)_K}. 
By applying cyclic permutations to ${}_{\Phi_K}(W)^2 \cdot W^2$, we obtain a genus-$(2h+n-1)$ Lefschetz fibration $f' : E(n)_K \to \mathbb{S}^2$ with the global monodromy ${}_{\Phi_K}(W) \cdot W \cdot W \cdot {}_{\Phi_K}(W)$, which is isomorphic as $f$ (see Section~\ref{LF}). 
By suitably cutting the base $\mathbb{S}^2$ of $f'$ into two disks $D_1$ and $D_2$, we get two genus-$(2h+n-1)$ Lefschetz fibrations $f_1 : X_1 \to D_1$ with the global monodromy ${}_{\Phi_K}(W) \cdot W$ and $f_2 : X_2 \to D_2$ with the global monodromy $W \cdot {}_{\Phi_K}(W)$. 
Since $W \cdot {}_{\Phi_K}(W)$ is obtained by applying cyclic permutations to ${}_{\Phi_K}(W) \cdot W$, $f_2$ is isomorphic to $f_1$. 
Here, since 
\begin{align*}
{}_{\Phi_K^{-1}}({}_{\Phi_K}(W) \cdot W) = W \cdot {}_{\Phi_K^{-1}}(W),
\end{align*}
the genus-$(2h+n-1)$ Lefschetz fibration $f'_1 : X'_1 \to \mathbb{D}^2$ with the global monodromy $W \cdot {}_{\Phi_K^{-1}}(W)$ is isomorphic to $f_1$ (see Section~\ref{LF}). 
Summarizing, $f_1'$ is isomorphic to $f_1$ and $f_2$, and hence $X_1'$ is diffeomorphic to $X_1$ and $X_2$.

By the definition of $\Phi_K (= \Phi_{T_{2,2h+1}})$ (see Section~\ref{LFonE(n)_K}), we see that 
\begin{align*}
\Phi_K^{-1} = t_{a_1} t_{a_2} \cdots t_{a_{2h-1}} t_{a_{2h}}
\end{align*}
in $\Gamma_{2h+n-1}$, where $a_1,a_2,\ldots,a_{2h}$ are the simple closed curves on $\Sigma_h^1 \subset \Sigma_{2h+n-1}$ as in Figure~\ref{gurtas}. 
Since $a_l$ is disjoint from $c_k$ for any $l,k$, we see that $\Phi_K^{-1}(c_k) = c_k$ (see, for example, Figure~\ref{gurtas}). 
Similarly, since $D_j$ is disjoint from $a_l$ for $j+2 \leq l$, we have $\Phi_K^{-1}(D_j) = t_{a_1}t_{a_2} \cdots t_{a_{j+1}}(D_j)$ (see the upper sides of Figures~\ref{gurtas2} and~\ref{gurtas3}). 
Moreover, it is easy to check that $t_{a_{j+1}}(D_j)$ is disjoint from $a_1,a_2,\ldots,a_j$, and hence $\Phi_K^{-1}(D_j) = t_{a_{j+1}}(D_j)$ for $j=0,1,\ldots,2h-1$ (see Figures~\ref{gurtas2} and~\ref{gurtas3}). 
From these observations, we obtain 
\begin{align*}
&W':={}_{\Phi_K^{-1}}(W) \\
&= t_{c_{2n-2}} \cdots t_{c_2} t_{c_1} t_{c_1} t_{c_2} \cdots t_{c_{2n-2}} t_{t_{a_1}(D_0)} t_{t_{a_2}(D_1)} \cdots t_{t_{a_{2h}}(D_{2h-1})} t_{\Phi_K^{-1}(D_{2h})} t_{c_{2n-1}}. 
\end{align*}
This means that the Lefschetz fibration $f$ in Lemma~\ref{lem:1001} is just $f_1'$.

Therefore, since $X_1'$ admits a handle decomposition without 1-handles from Lemma~\ref{lem:1001}, 
$X_1$ and $X_2$ also admit such a handle decomposition, and hence, Theorem~\ref{thm:1} immediately follows from Lemma~\ref{lem:1000}. 
\end{proof}

In the rest of the section, we show Lemma~\ref{lem:1001}. 
%To prove this, we prepare Lemmas~\ref{lem:1002} and~\ref{lem:1003} below. 
%We provide preliminaries to state them. 
%
%
%
%We fix the subsurface $\Sigma_{2h+n-1}^1$ obtained by removing the $2$-dimensional 2-handle from the fixed handle decomposition of $\Sigma_{2h+n-1}$ as in Figure~\ref{gurtashandledecomp}. 
Let The $4$-dimensional $\alpha_1^\ast,\beta_1^\ast,\alpha_2^\ast,\beta_2^\ast,\ldots, \alpha_{2h+n-1}^\ast,\beta_{2h+n-1}^\ast$ are as in Figure~\ref{gurtas1handles}.

Let $c_1,c_2,\ldots,c_{2n-1},D_0,D_1,\ldots,D_{2h},t_{a_1}(D_0),t_{a_2}(D_1),\ldots,t_{a_{2h}}(D_{2h-1})$ be simple closed curves on $\Sigma_{2h+n-1}$ as in Figures~\ref{gurtas},~\ref{gurtas2} and~\ref{gurtas3}. 
We take a small open disk $D$ on $\Sigma_{2h+n-1}$ as in this figure. 
Let us consider the Kirby diagram of $X$ ``on $\Sigma_{2h+n-1}-D$'' described in Section~\ref{KirbydiagramS}. 
We set 
\begin{align*}
\mathcal{V} = \{c_{2n-2},\ldots,c_2,c_1,D_0,D_1,\ldots,D_{2h}, t_{a_1}(D_0),t_{a_2}(D_1),\ldots,t_{a_{2h}}(D_{2h-1})\}
\end{align*}
which is the set vanishing cycles of the genus-$(2h+n-1)$ Lefschetz fibration $f: X \to \mathbb{D}^2$ in Lemma~\ref{lem:1001} without overlapping.

We now prove Lemma~\ref{lem:1001}. 
\begin{proof}[Proof of Lemma~\ref{lem:1001}]
We consider the case of $h=3$. 
%Recall that our handle decomposition of $\Sigma_{2h+n-1}^1$ is as in Figure~\ref{gurtashandledecomp}. 
Let $h_v$ be the corresponding 2-handle of a vanishing cycle $v$ in $\mathcal{V}$.

Let us consider a subdiagram $\mathcal{D}_0$ of the Kirby diagram on $\Sigma_{6+n-1}^1$ of $X$ arising from the Lefschetz fibration $f$ in Lemma~\ref{lem:1000} whose 2-handles are $h_{c_1}, h_{c_2}, \ldots, h_{c_{2n-2}}$, $h_{D_0}, h_{D_1},\ldots,h_{D_6}$ and $h_{t_{a_{1}}(D_0)}, h_{t_{a_{2}}(D_1)}, \ldots, h_{t_{a_6}(D_5)}$ (see Figures~\ref{2handles21} and~\ref{2handles102}). 
Note that $h_{c_{2n-2}}$ is closer to the reader than $h_{t_{a_{i+1}}(D_i)}$ and $h_{D_j}$. 
We slide $h_{t_{a_{i+1}}(D_i)}$ over $h_{D_i}$ for $i=0,1,\ldots,5$ as in Figures~\ref{2handles31} and~\ref{2handles103}, and let $h_{i+1}$ be the resulting 2-handle.
An isotopy gives Figures~\ref{2handles41} and~\ref{2handles104}. 
Then, the resulting diagram $\mathcal{D}_1$ has 2-handles $h_{c_1},h_{c_2},\ldots,h_{c_{2n-2}}$, $h_{D_0}, h_{D_1},\ldots,h_{D_6}$ and $h_{1},h_{2},\ldots,h_6$. 
We next slide $h_{D_{6-j}}$ over $h_{D_{5-j}}$ for $j=0,1,\ldots,5$ as in Figures~\ref{2handles61} and~\ref{2handles106} (for convenience, see Figures~\ref{2handles51} and~\ref{2handles105}, which have the framings of $h_{D_0},h_{D_1},\ldots,h_{D_5}$), and let $h_{5-j,6-j}$ be the resulting 2-handle. 
An isotopy gives Figures~\ref{2handles71} and~\ref{2handles107}. 
Then, the resulting diagram $\mathcal{D}_3$ has $h_{D_0}$, $h_{c_1},h_{c_2},\ldots,h_{c_{2n-2}}$, $h_{1},h_{2},\ldots,h_6$ and $h_{0,1},h_{1,2},\ldots,h_{5,6}$. 
Since in the rest of the proof, we do not use the 2-handle $h_{D_0}$, we consider the diagram $\mathcal{D}_4$ removed $h_{D_0}$ from $\mathcal{D}_3$ (see Figures~\ref{2handles81} and~\ref{2handles108}). 
An isotopy gives Figures~\ref{2handles91} and~\ref{2handles109}.

We now apply handle canceling operations.

The first step is to remove the 1-handles $\alpha^\ast_k$ and $\beta^\ast_i$ for $k=6+1,6+2,\ldots,6+n-1$ ($2h=6$).  
The 2-handles $h_{c_1},h_{c_2},\ldots,h_{c_{2n-2}}$ are as in the left part of Figure~\ref{2handles1003}. 
An isotopy gives the right part of Figure~\ref{2handles1003}. 
Since the pair $(\alpha_{6+n-1}^\ast, h_{c_1})$ is a canceling pair, we remove it. 
Then, the pair $(\alpha_{6+n-2}^\ast, h_{c_3})$ becomes a canceling pair since the 1-handle $\alpha^\ast_{6+n-1}$ was removed, and hence we remove $(\alpha_{6+n-2}^\ast, h_{c_3})$. 
Similarly, the pair $(\alpha_{6+n-3}^\ast, h_{c_5})$ also becomes a canceling pair since we do not have the 1-handle $\alpha^\ast_{6+n-2}$, and hence we remove $(\alpha_{6+n-3}^\ast, h_{c_5})$. 
By repeating this argument, the pairs $(\alpha_{6+n-1},h_{c_1}), (\alpha_{6+n-2},h_{c_3}),\ldots,(\alpha_{6+1},h_{c_{2n-3}})$ are removed. 
Note that the pairs $(\alpha_{6+n-1}^\ast, h_{c_2})$, $(\alpha_{6+n-2}^\ast, h_{c_4}),\ldots,(\alpha_{6+1}^\ast,h_{c_{2n-2}})$ are canceling pairs, we remove them.

In the rest of the proof, we refer to Figures~\ref{2handles91} and~\ref{2handles109}.

The second step is to remove the 1-handles $\alpha^\ast_1,\alpha^\ast_2,\alpha^\ast_3$ ($h=3$). 
We see that the pair $(\alpha^\ast_1, h_1)$ is a canceling pair by an isotopy (see, for example, Figure~\ref{2handles92}), and therefore, we remove it. 
Then, the pair $(\alpha^\ast_2,h_3)$ becomes a canceling pair by choosing an isotopy. 
Similarly, the pair $(\alpha^\ast_3,h_5)$ becomes a canceling pair after removing $(\alpha^\ast_2,h_3)$.

The third step is to remove the 1-handles $\beta^\ast_1,\beta^\ast_2,\beta^\ast_3$ ($h=3$). 
For each $i=1,2,3$, by choosing an isotopy, the pair $(\beta^\ast_i,h_2i)$ becomes a canceling pair. 
Therefore, we remove $(\beta^\ast_1,h_2), (\beta^\ast_2,h_4), (\beta^\ast_3,h_6)$.

The fourth step is to remove the 1-handles $\alpha^\ast_4, \alpha^\ast_5, \alpha^\ast_6$ ($2h=6$). 
Since the 1-handle $\alpha^\ast_1$ was already removed, the pair $(\alpha^\ast_6, h_{0,1})$ becomes a canceling pair by an isotopy. 
Therefore, we remove it, and then, the pair $(\alpha^\ast_5, h_{2,3})$ becomes a canceling pair by choosing an isotopy since the 1-handles $\alpha^\ast_1,\alpha_2,\alpha_6$ were already removed. 
Similarly, the pair $(\alpha^\ast_4, h_{4,5})$ becomes a canceling pair by an isotopy since the 1-handles $\alpha^\ast_2,\alpha^\ast_3,\alpha^\ast_5$ were removed.

The final step is to remove the 1-handles $\beta^\ast_4, \beta^\ast_5, \beta^\ast_6$ ($2h=6$). 
Since the 1-handle $\beta^\ast_i$ was already removed for $i=1,2,3$, each pair $(\beta^\ast_{7-i},h_{2i-1,2i})$ becomes a canceling pair by an isotopy. 
Therefore, we remove the pairs $(\beta^\ast_4,h_{5,6}), (\beta^\ast_3,h_{3,4}), (\beta^\ast_4,h_{1,2})$.

From the argument above, we obtain a diagram without 1-handles from $\mathcal{D}_3$. 
We note that a Lefschetz fibration over the disk has a handle decomposition without 3- and 4-handles (see Section~\ref{Kirbydiagram}). 
The proof for general $h$ is similar. 
This finishes the proof of Lemma~\ref{lem:1000}. 
\end{proof}

\section*{Acknowledgments} The first author was supported by JSPS KAKENHI Grant Numbers JP20K03613. 
The second author was supported by the JSPS OU-SPRING, and the Public Interest Incorporated Foundation "Ohmoto ikueikai" in carrying out this research.
The authors would like to thank Kenta Hayano, Motoo Tange and Kouichi Yasui for their comments on an earlier version of this paper.

%\newpage

\begin{figure}[hbt]
  \centering
       \includegraphics[scale=.65]{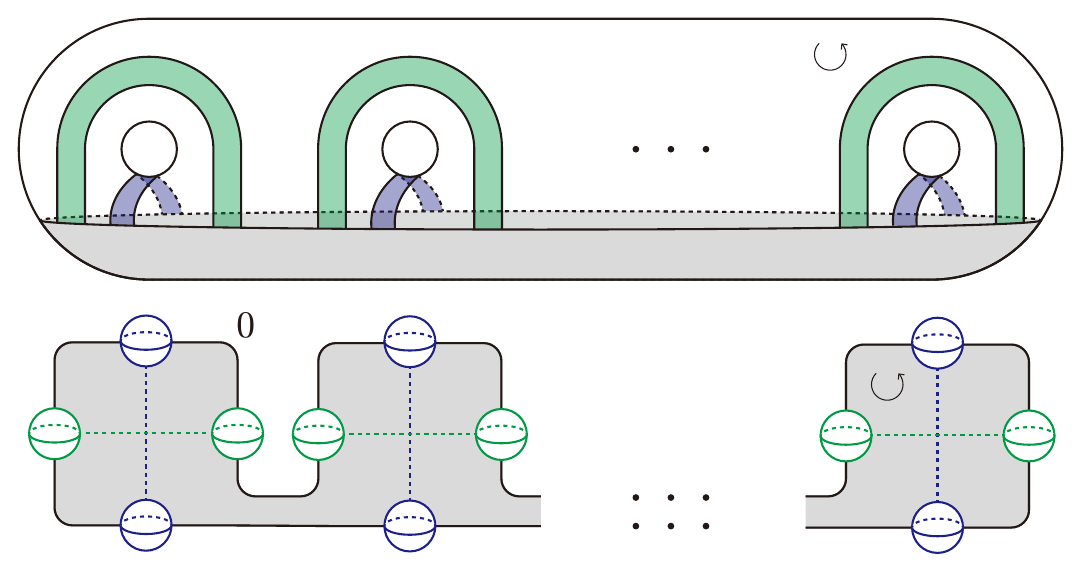}
       \caption{A handle decomposition of $\Sigma_g$ and a Kirby diagram of $\Sigma_g \times \mathbb{D}^2$.}
       \label{HDandKD1}
  \end{figure}

\begin{figure}[hbt]
  \centering
       \includegraphics[scale=.65]{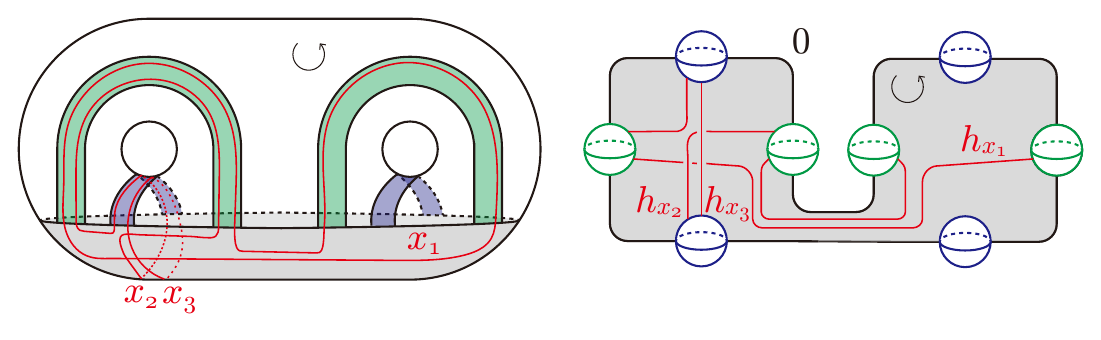}
       \caption{The vanishing cycles $x_1,x_2,x_3$ of the genus-$2$ Lefschetz fibration $f_0$ on $X_0$ and a Kirby diagram of $X_0$.}
       \label{HDandKD2}
  \end{figure}

%\newpage
  
\begin{figure}[hbt]
  \centering
       \includegraphics[scale=.65]{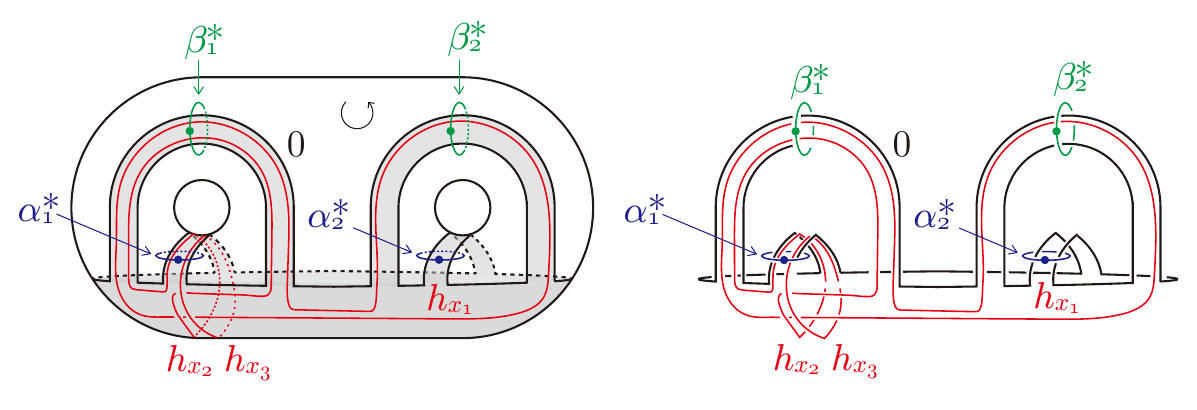}
       \caption{RIGHT: a Kirby diagram ``on $\Sigma_g^1$" of $X_0$, LEFT: a standard Kirby diagram of $X_0$.}
       \label{HDandKD3}
  \end{figure}

\begin{figure}[hbt]
  \centering
       \includegraphics[scale=.65]{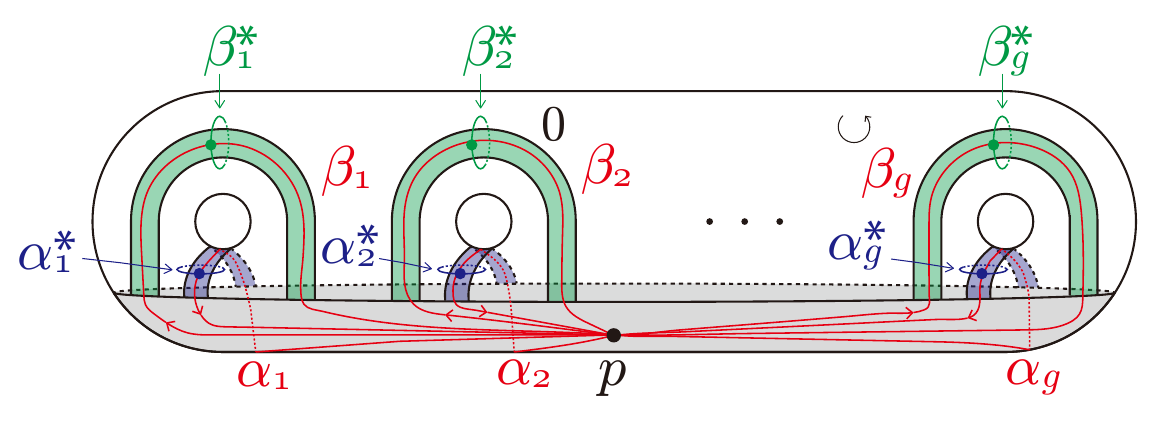}
       \caption{The generators $\alpha_i,\beta_i$ of $\pi_1(\Sigma_g^1,p)$, and the 1-handles $\beta_i^\ast,\alpha_i^\ast$ for $i=1,2,\ldots,g$.}
       \label{1handles}
  \end{figure}

\begin{figure}[hbt]
  \centering
       \includegraphics[scale=.60]{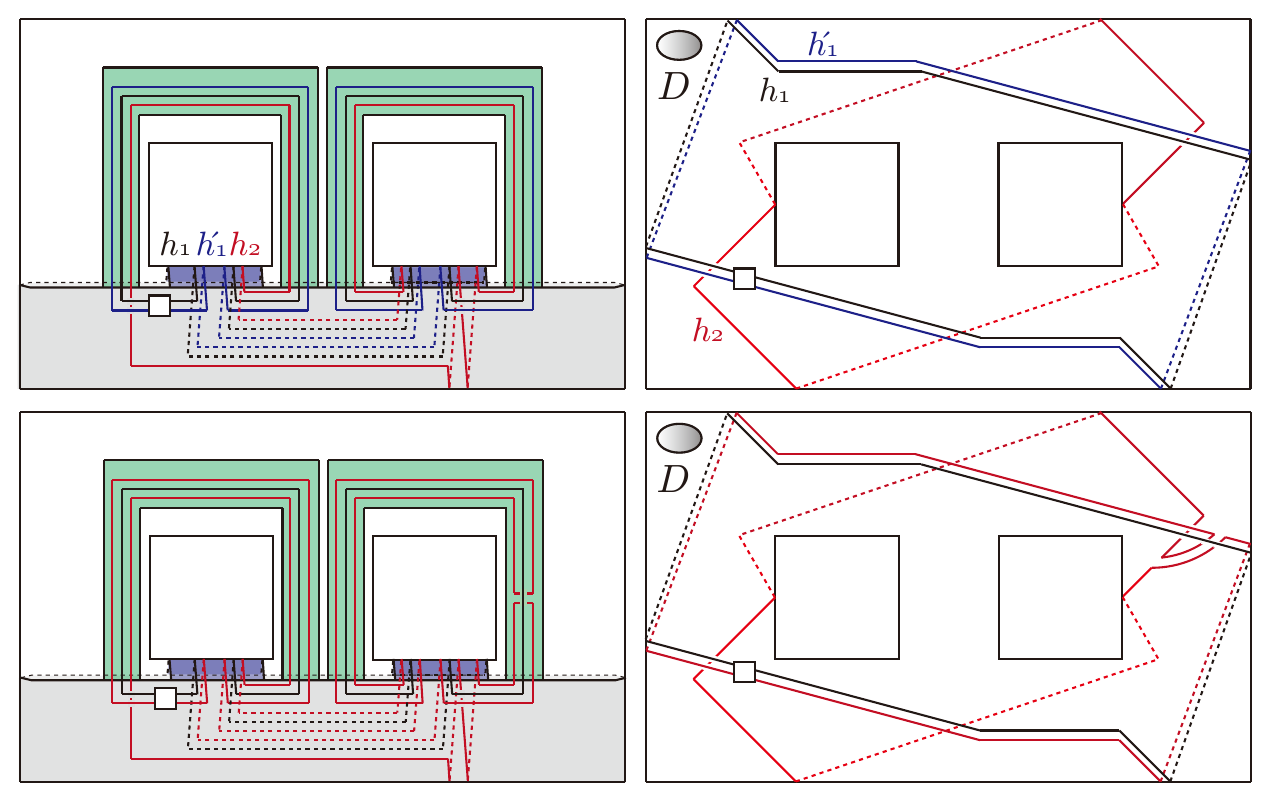}
       \caption{}
       \label{diagram3}
  \end{figure}
  
\begin{figure}[hbt]
  \centering
       \includegraphics[scale=.60]{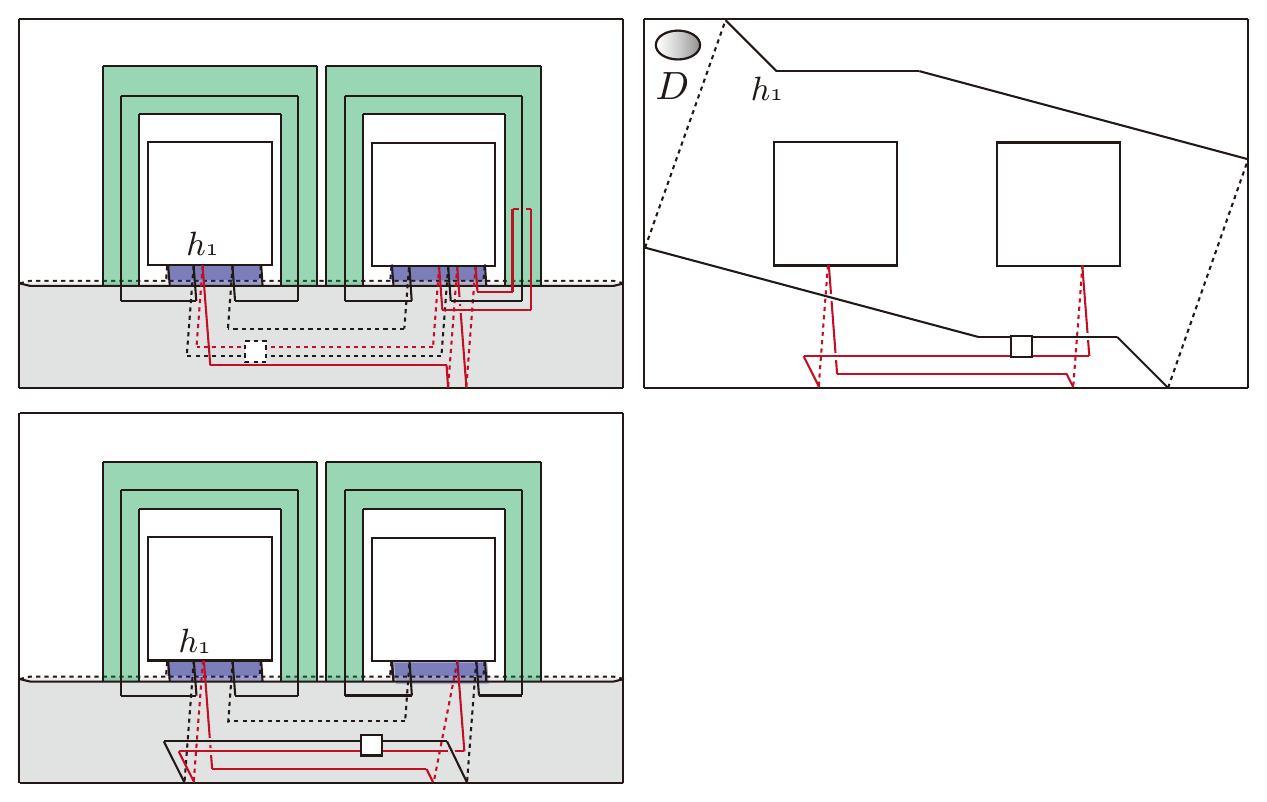}
       \caption{}
       \label{diagram4}
  \end{figure}

\begin{figure}[hbt]
  \centering
       \includegraphics[scale=.70]{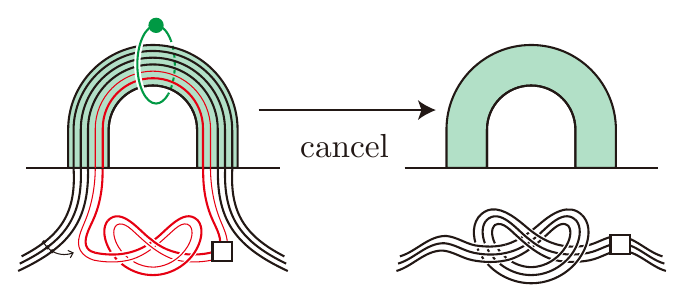}
       \caption{A canceling operation on a Kirby diagram on $\Sigma_g^1$.}
       \label{cancelingoperation1}
  \end{figure}

\begin{figure}[hbt]
  \centering
       \includegraphics[scale=.54]{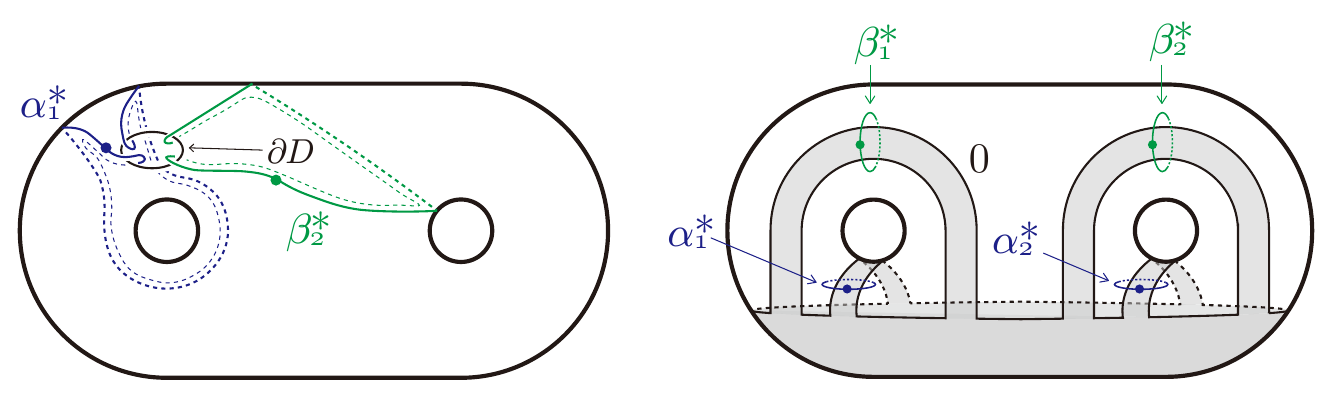}
       \caption{THE LEFT: the deformed 1-handles $\alpha^\ast_1,\beta^\ast_2$, THE RIGHT: the 1-handles $\alpha^\ast_1,\alpha^\ast_2,\beta^\ast_1,\beta^\ast_2$.}
       \label{deform1handle}
  \end{figure}

\

\begin{figure}[hbt]
  \centering
       \includegraphics[scale=.51]{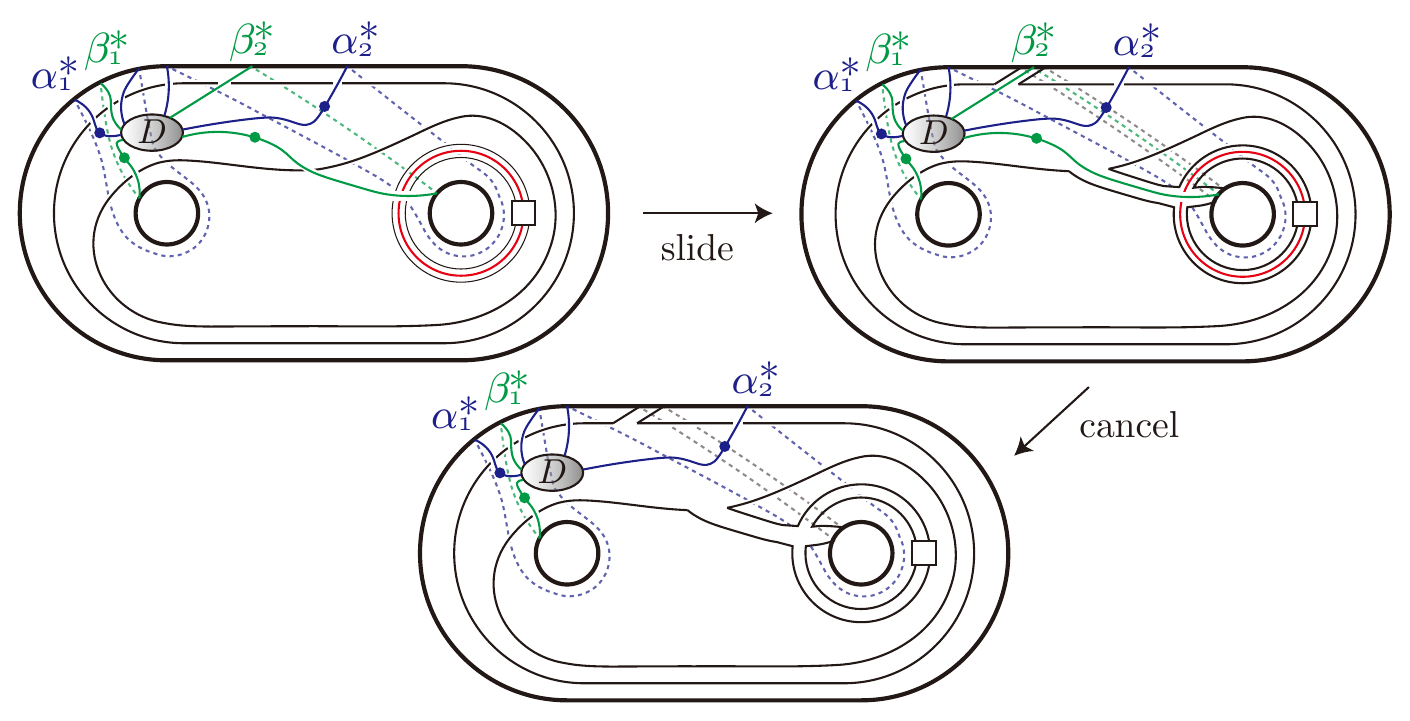}
       \caption{A canceling operation.}
       \label{cancelingoperation0}
  \end{figure}

\begin{figure}[hbt]
  \centering
       \includegraphics[scale=.51]{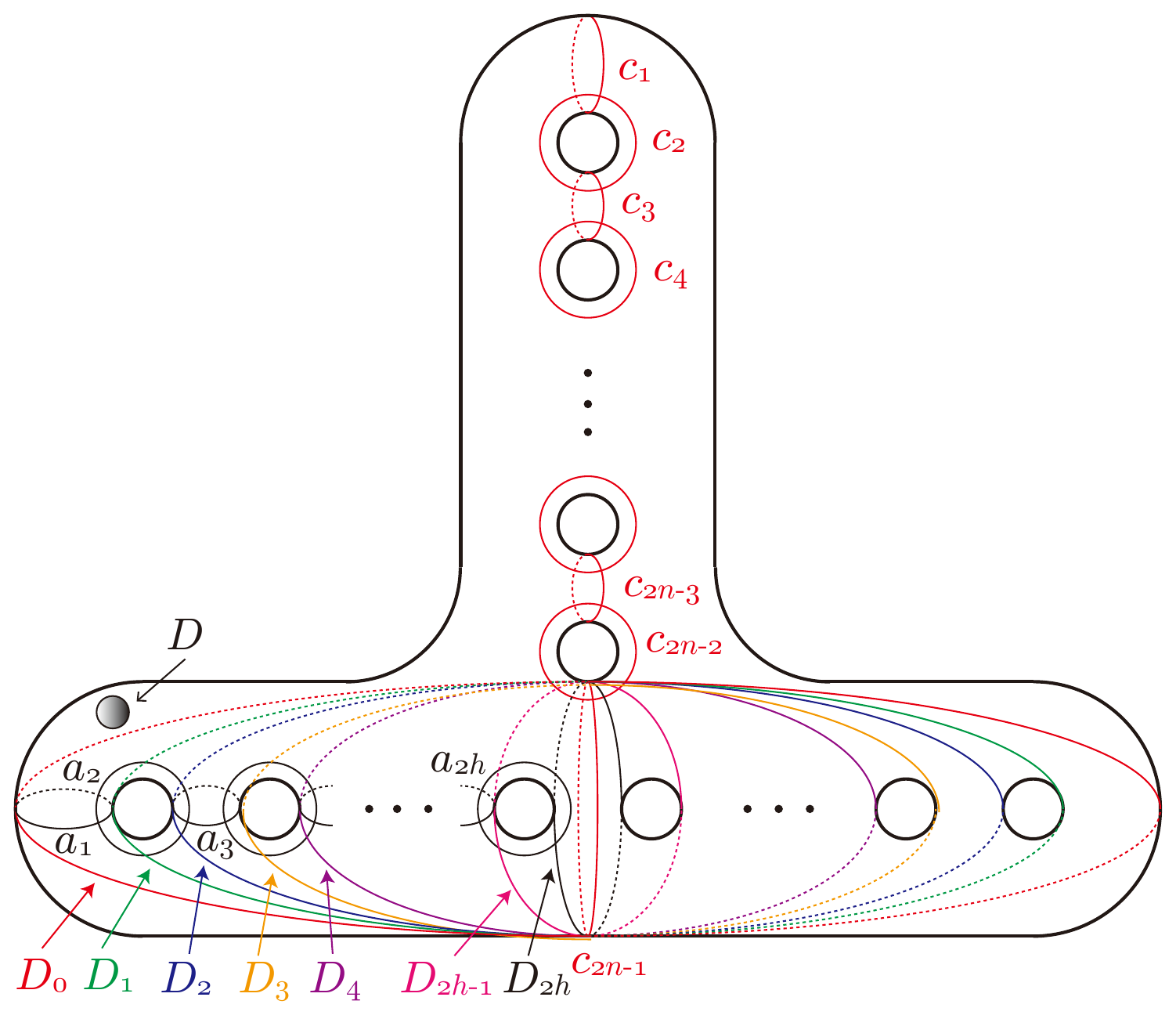}
       \caption{The simple closed curves $c_i,D_j,a_k$ on $\Sigma_{2h+n-1}$ for $i=1,2,\ldots,2n-1$, $j=0,1,\ldots,2h$ and $k=1,2,\ldots,2h$, and the small open disk $D$ on $\Sigma_{2h+n-1}$.}
       \label{gurtas}
  \end{figure}

\begin{figure}[hbt]
  \centering
       \includegraphics[scale=.51]{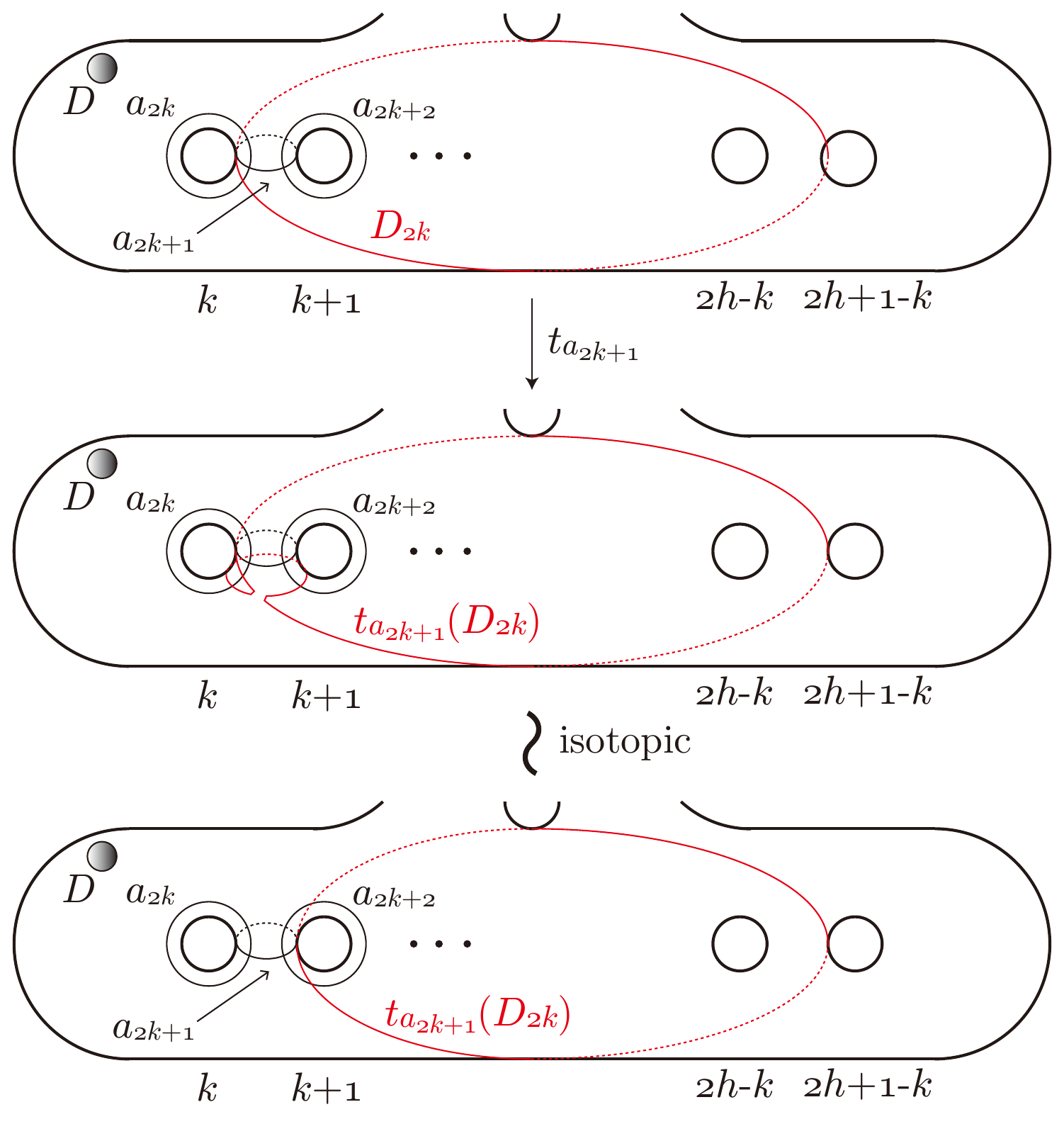}
       \caption{The simple closed curves $D_{2k}$ and $\Phi_K(D_{2k}) = \Phi_{T_{2,2h+1}}(D_{2k}) = t_{a_{2k+1}(D_{2k})}$.}
       \label{gurtas2}
  \end{figure}

\begin{figure}[hbt]
  \centering
       \includegraphics[scale=.51]{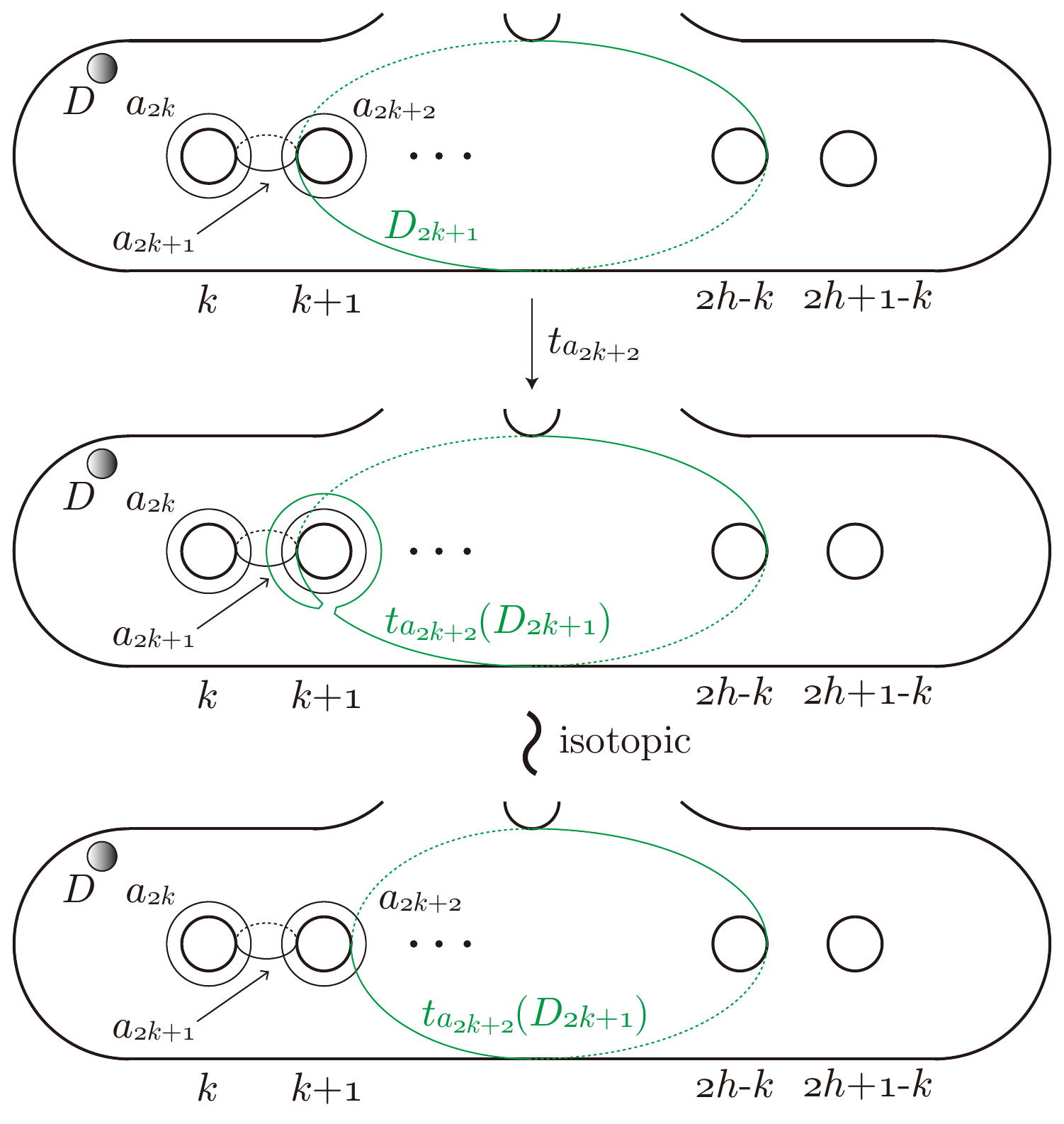}
       \caption{The simple closed curves $D_{2k+1}$ and $\Phi_K(D_{2k+1}) = \Phi_{T_{2,2h+1}}(D_{2k+1}) = t_{a_{2k+2}(D_{2k+1})}$.}
       \label{gurtas3}
  \end{figure}

\begin{figure}[hbt]
  \centering
       \includegraphics[scale=.65]{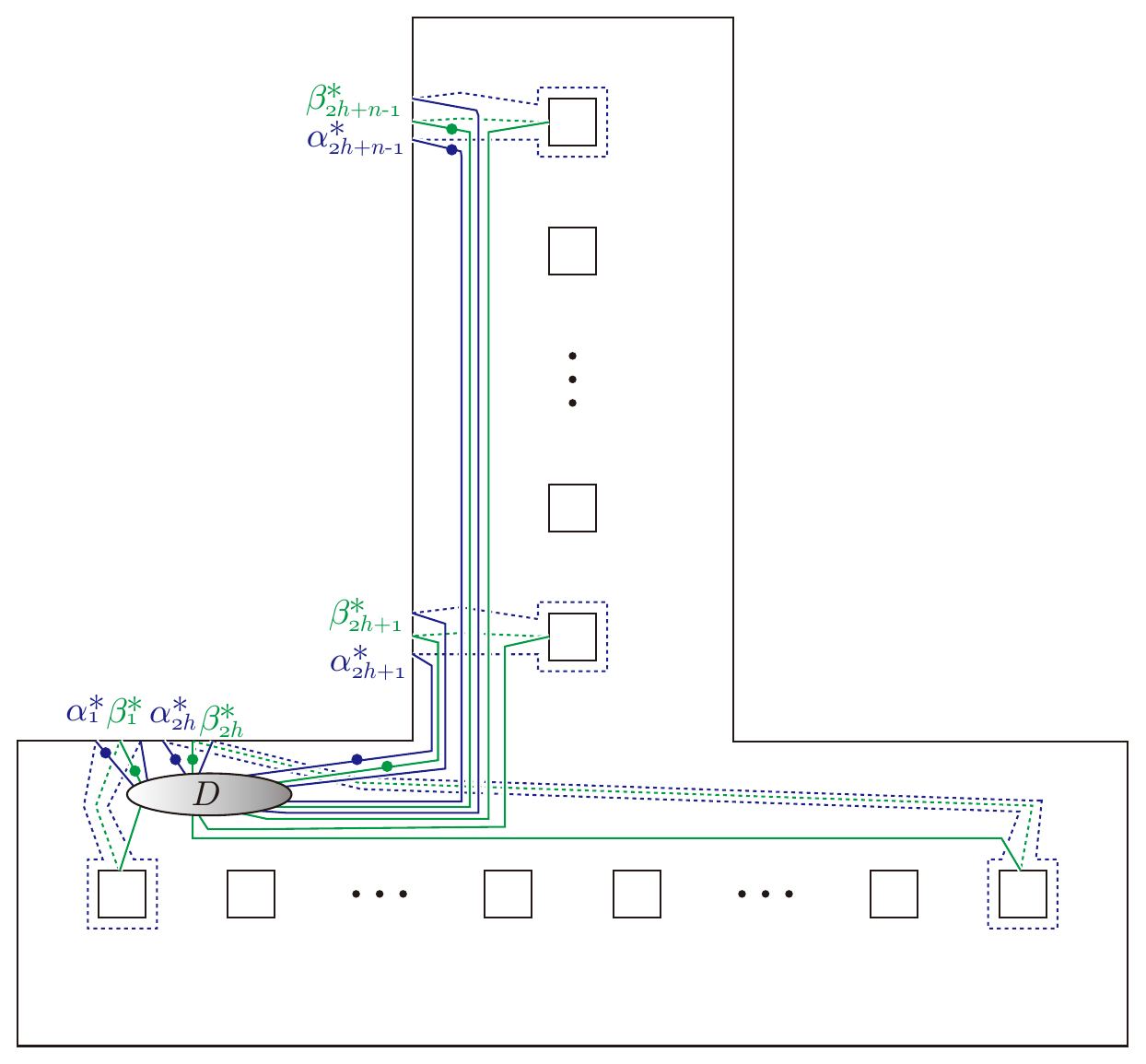}
       \caption{$1$-handles $\alpha_1^\ast,\beta_1^\ast,\alpha_2^\ast,\beta_2^\ast,\ldots, \alpha_{2h+n-1}^\ast,\beta_{2h+n-1}^\ast$.}
       \label{gurtas1handles}
  \end{figure}

\begin{figure}[hbt]
  \centering
       \includegraphics[scale=.70, angle=90]{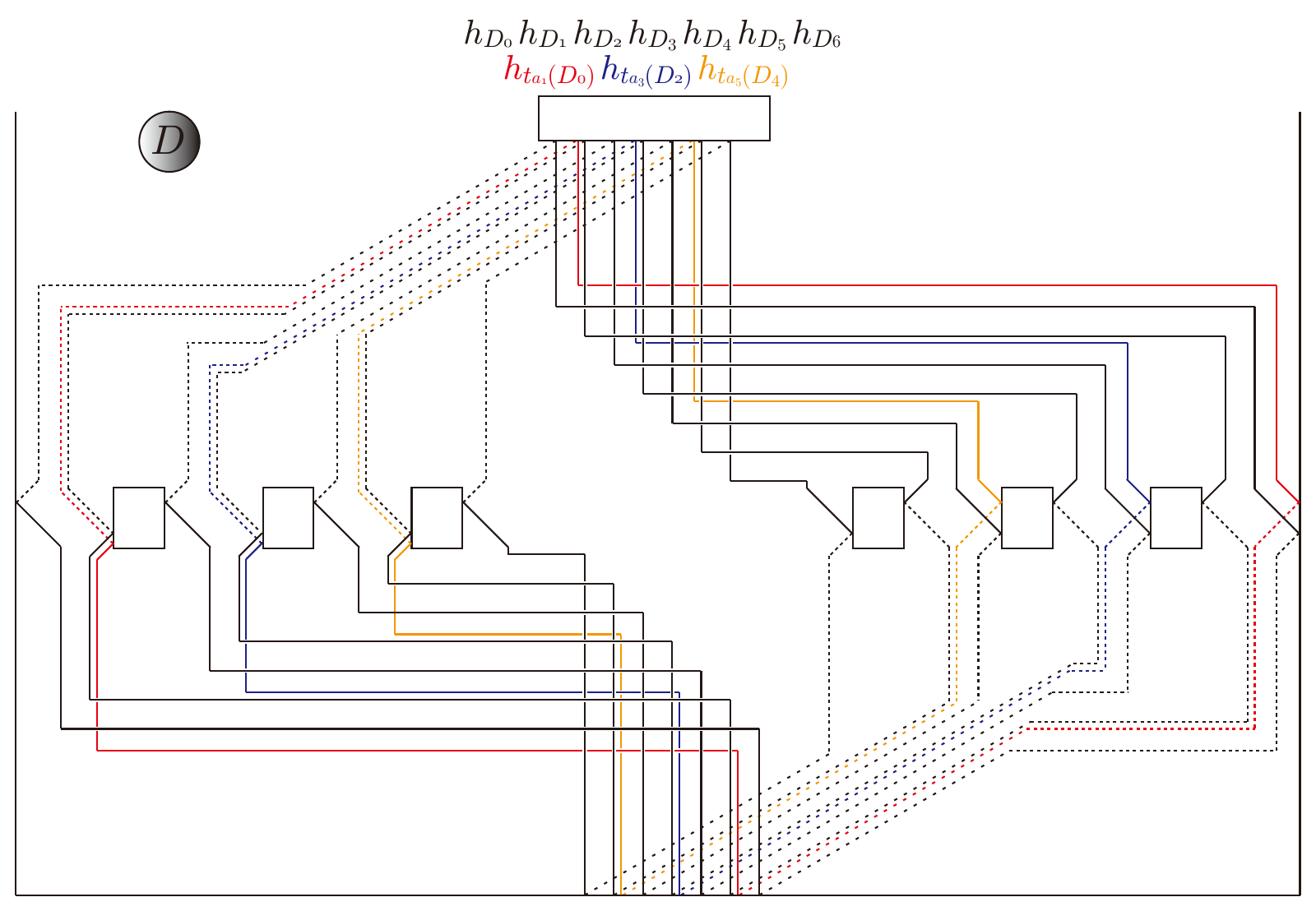}
       \caption{}
       \label{2handles21}
  \end{figure}

\begin{figure}[hbt]
  \centering
       \includegraphics[scale=.70, angle=90]{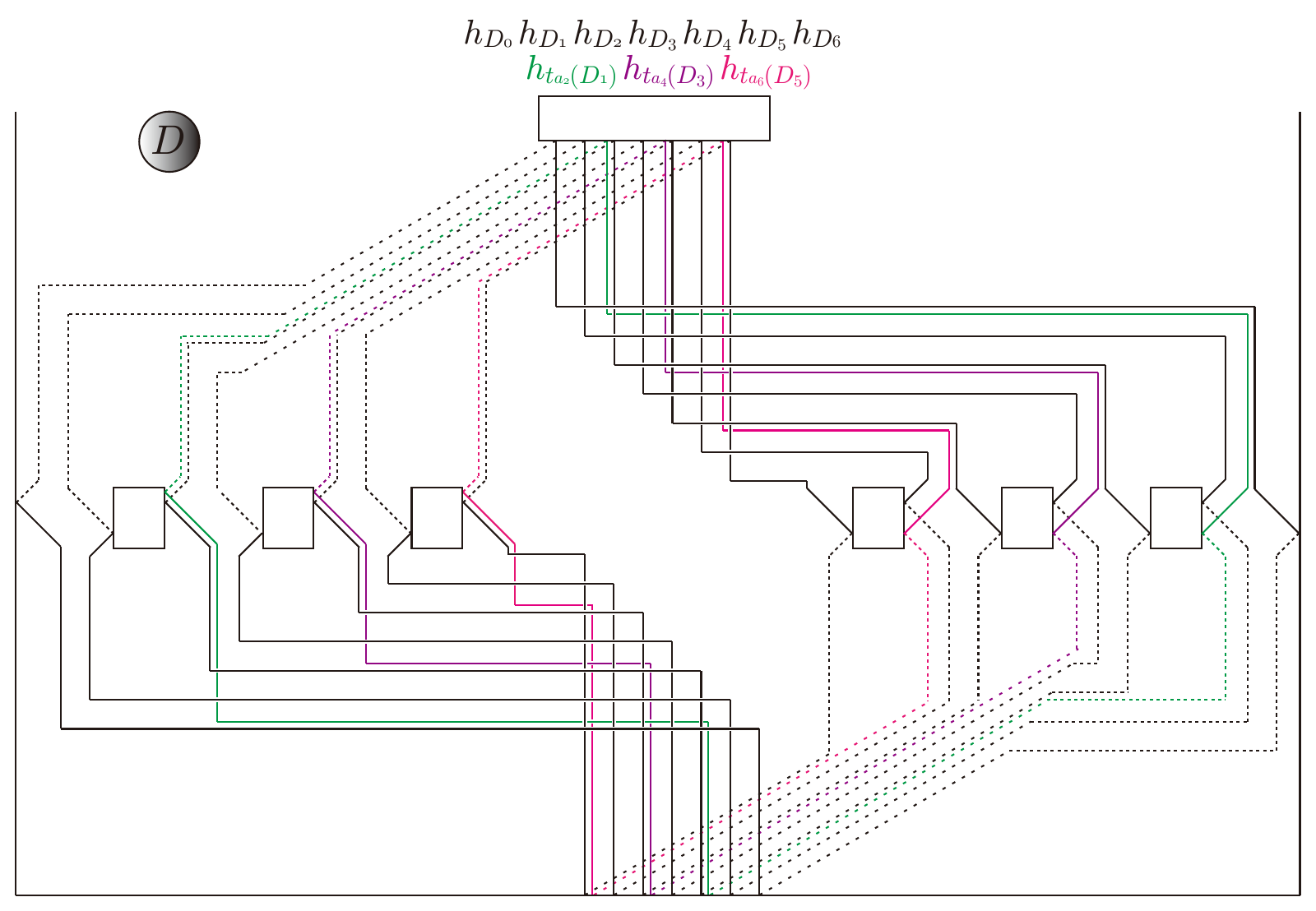}
       \caption{}
       \label{2handles102}
  \end{figure}

\begin{figure}[hbt]
  \centering
       \includegraphics[scale=.70, angle=90]{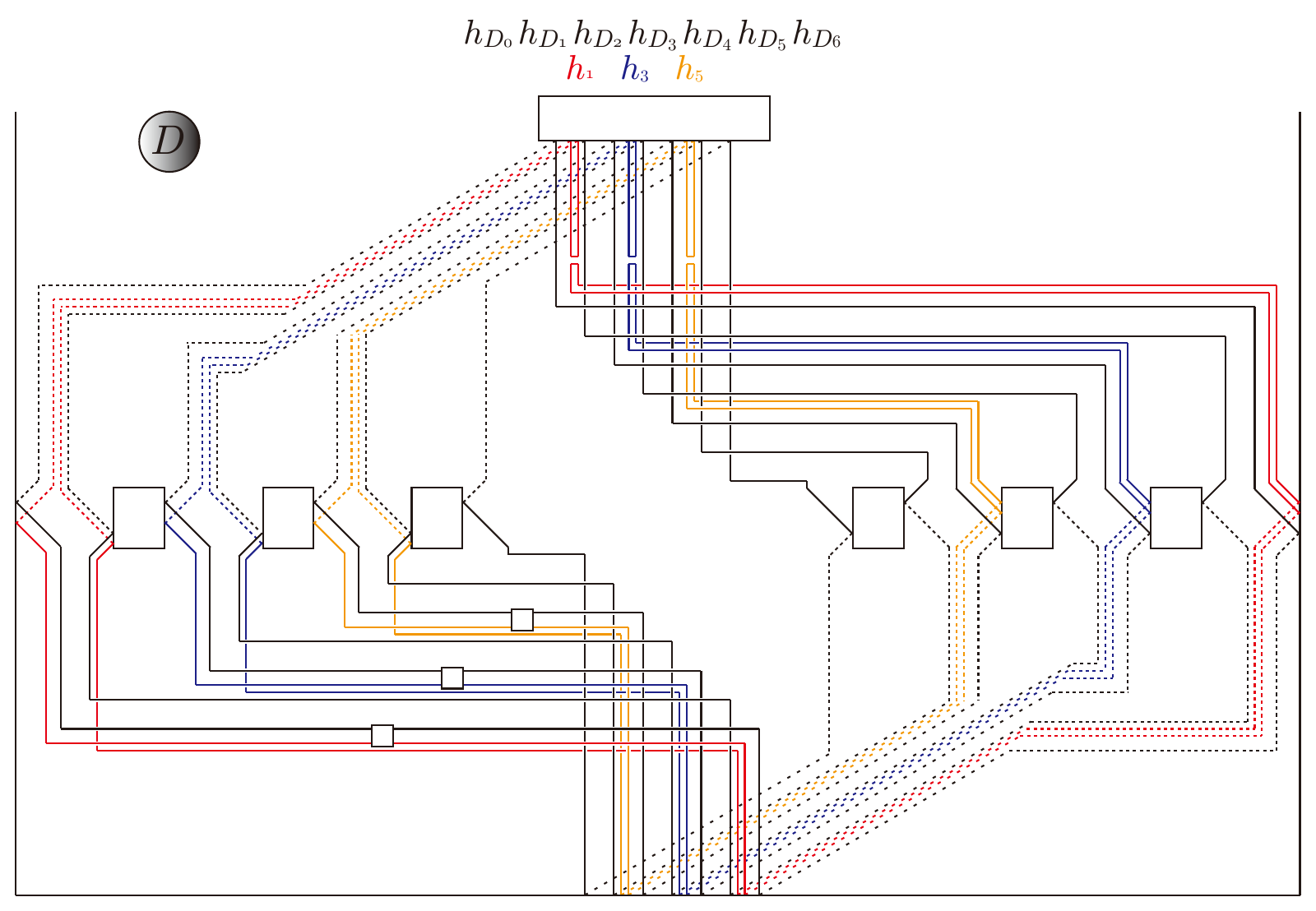}
       \caption{}
       \label{2handles31}
  \end{figure}

\begin{figure}[hbt]
  \centering
       \includegraphics[scale=.70, angle=90]{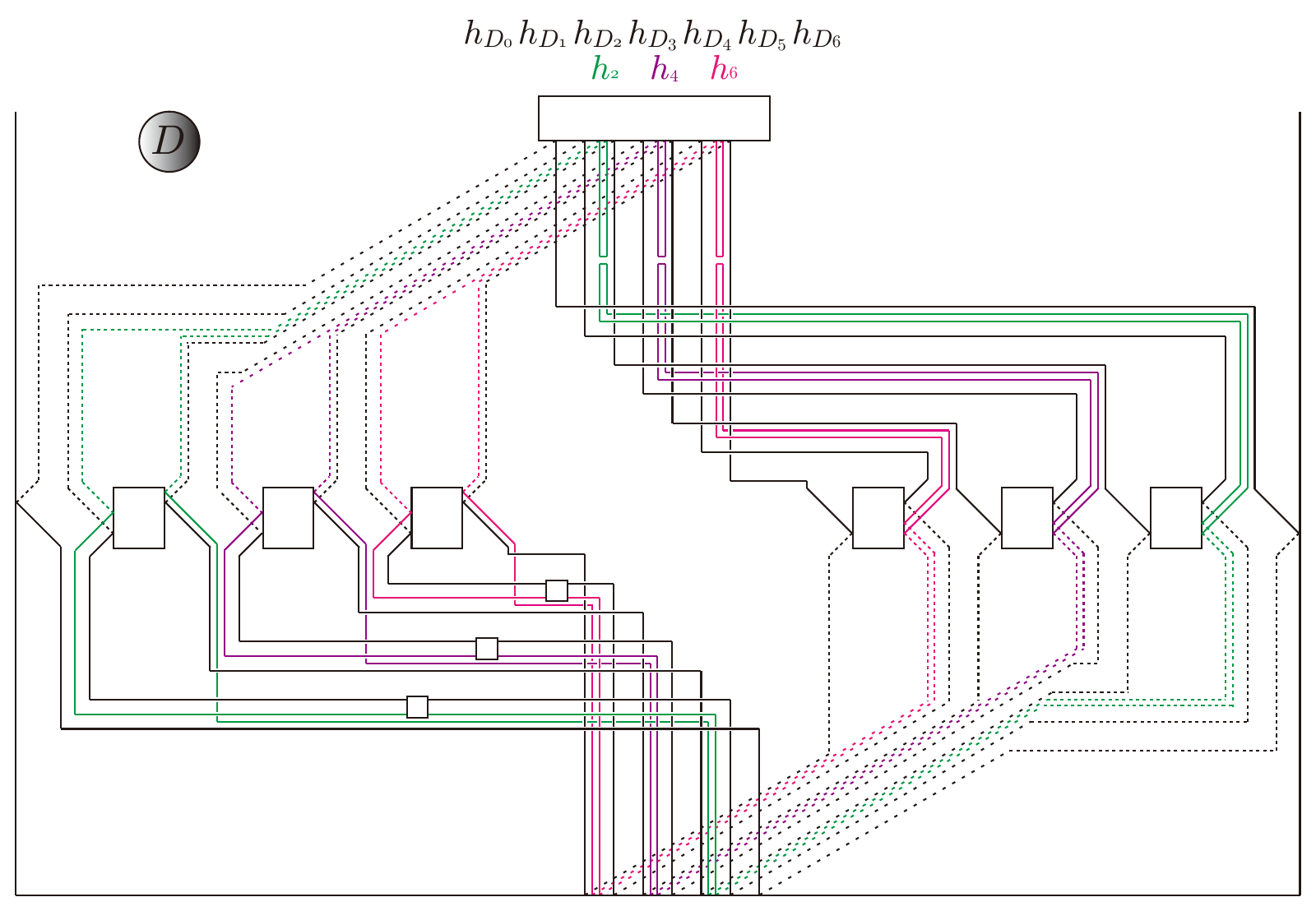}
       \caption{}
       \label{2handles103}
  \end{figure}

\begin{figure}[hbt]
  \centering
       \includegraphics[scale=.70, angle=90]{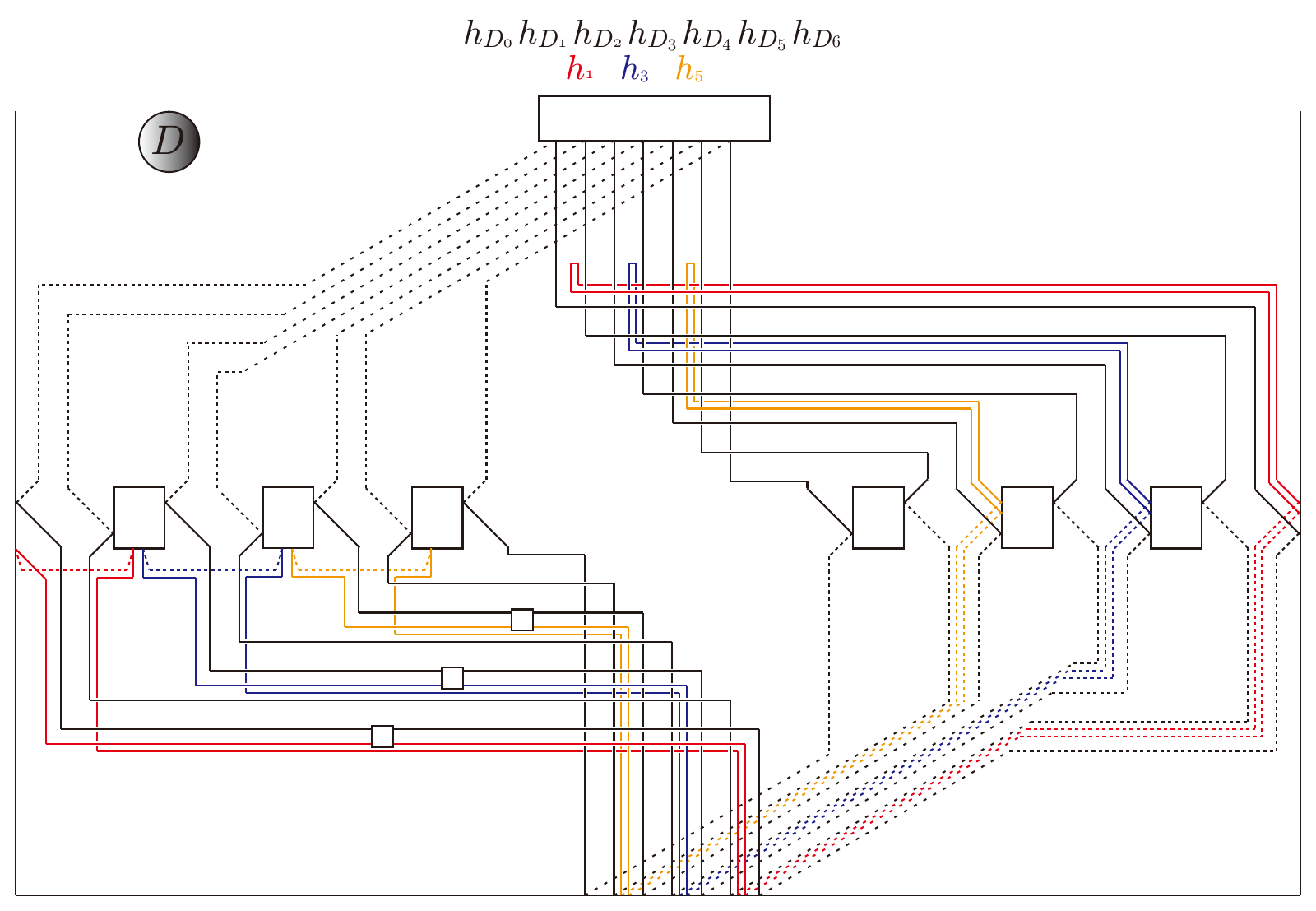}
       \caption{}
       \label{2handles41}
  \end{figure}

\begin{figure}[hbt]
  \centering
       \includegraphics[scale=.70, angle=90]{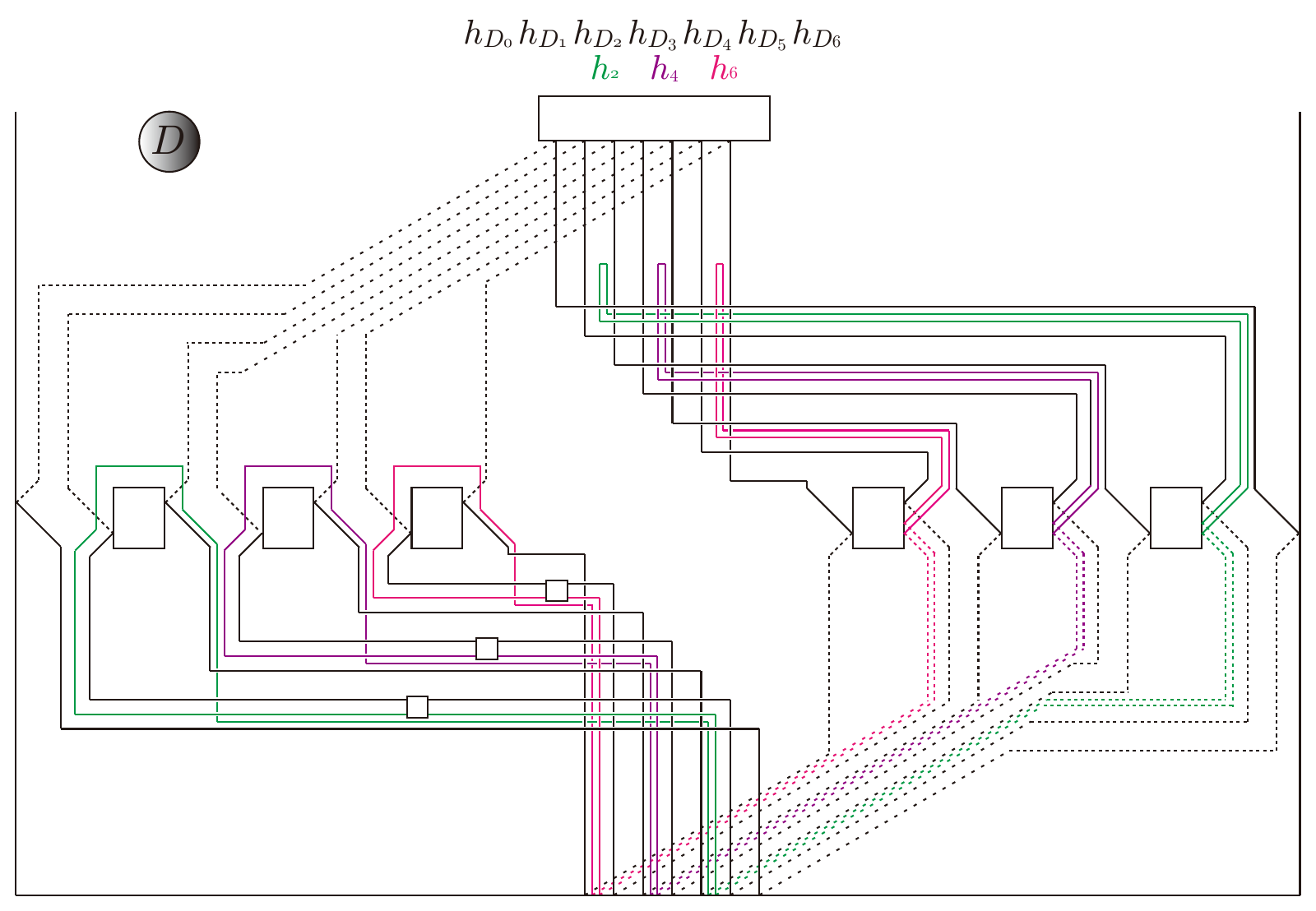}
       \caption{}
       \label{2handles104}
  \end{figure}

\begin{figure}[hbt]
  \centering
       \includegraphics[scale=.70, angle=90]{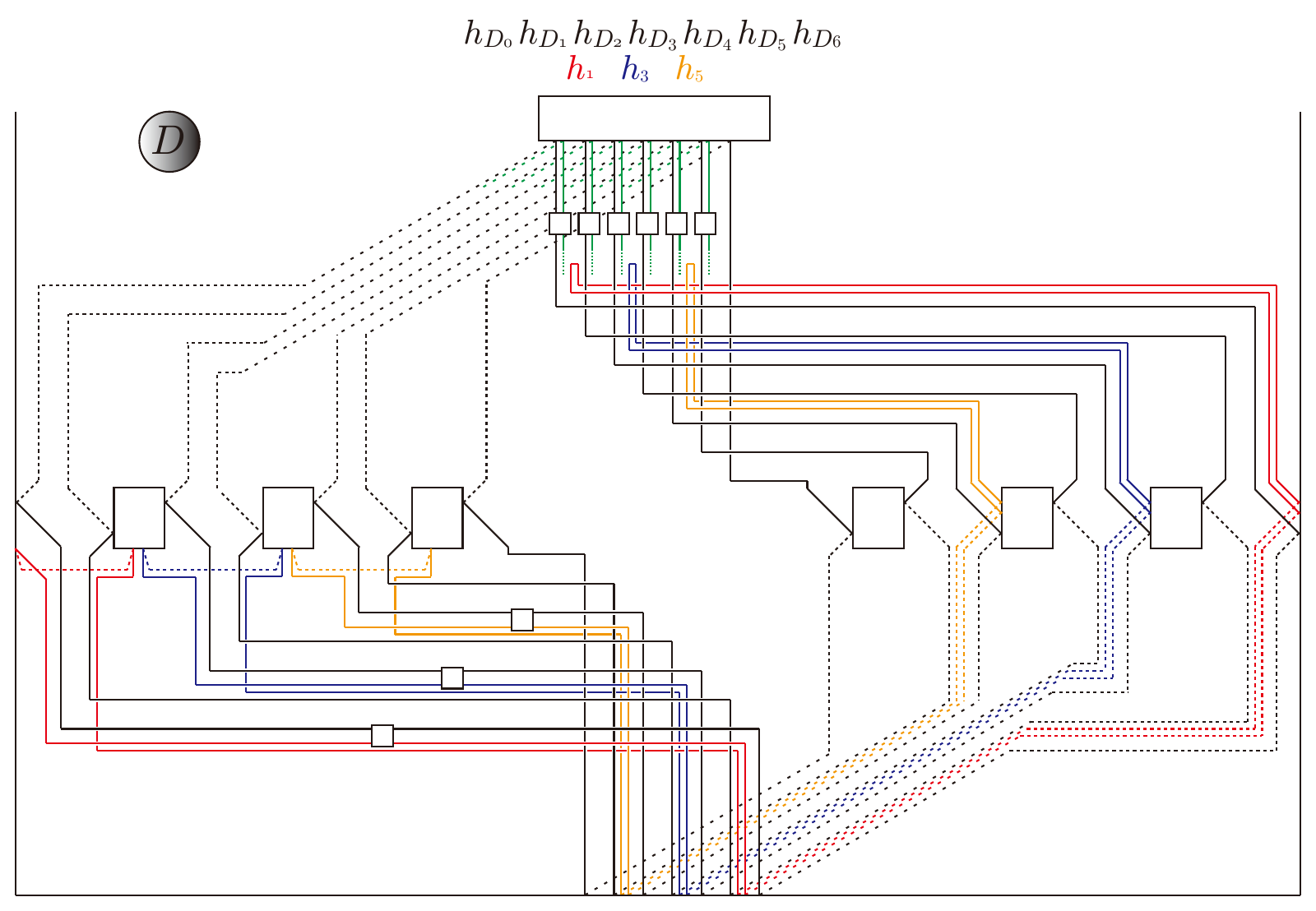}
       \caption{}
       \label{2handles51}
  \end{figure}

\begin{figure}[hbt]
  \centering
       \includegraphics[scale=.70, angle=90]{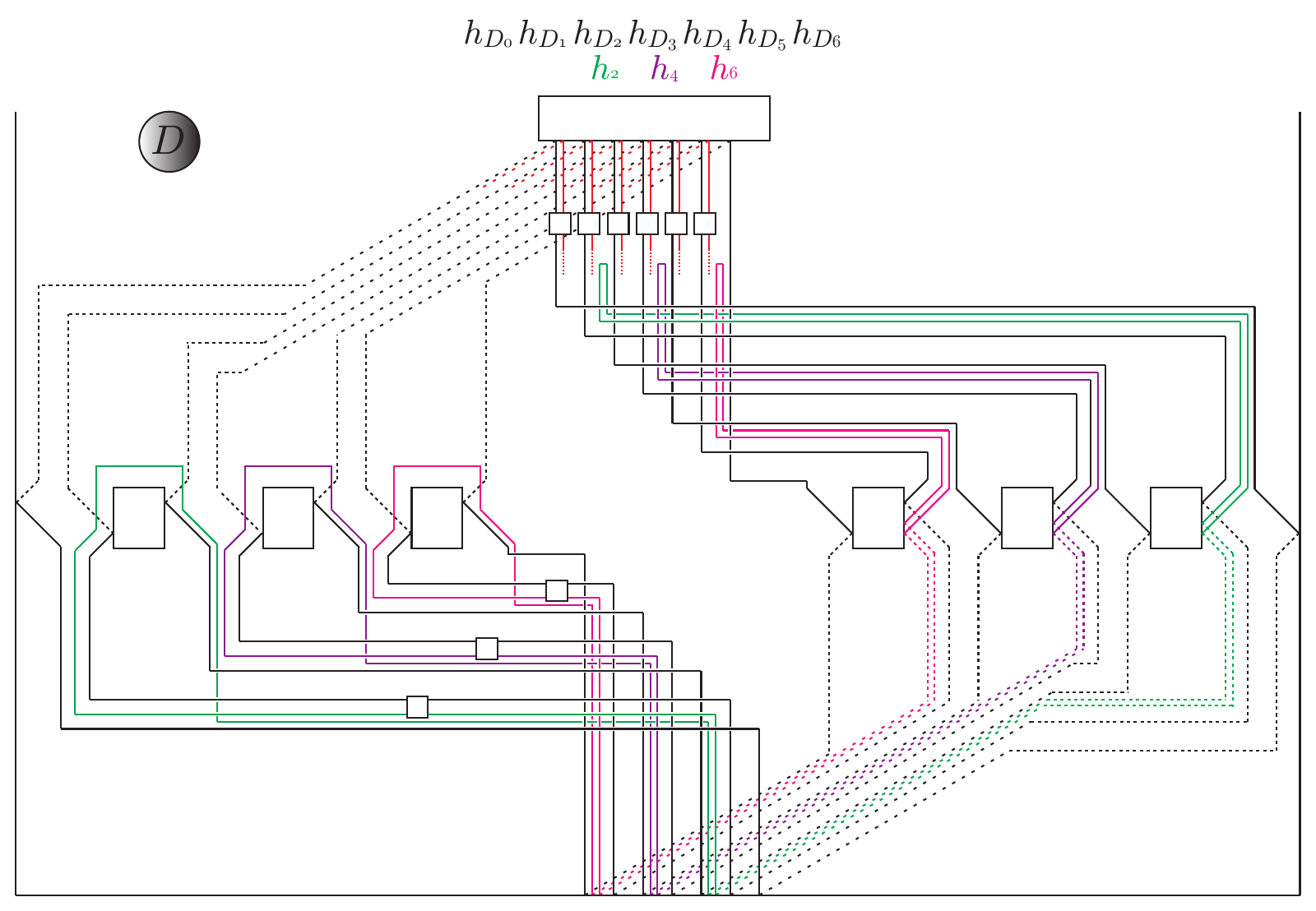}
       \caption{}
       \label{2handles105}
  \end{figure}

\begin{figure}[hbt]
  \centering
       \includegraphics[scale=.70, angle=90]{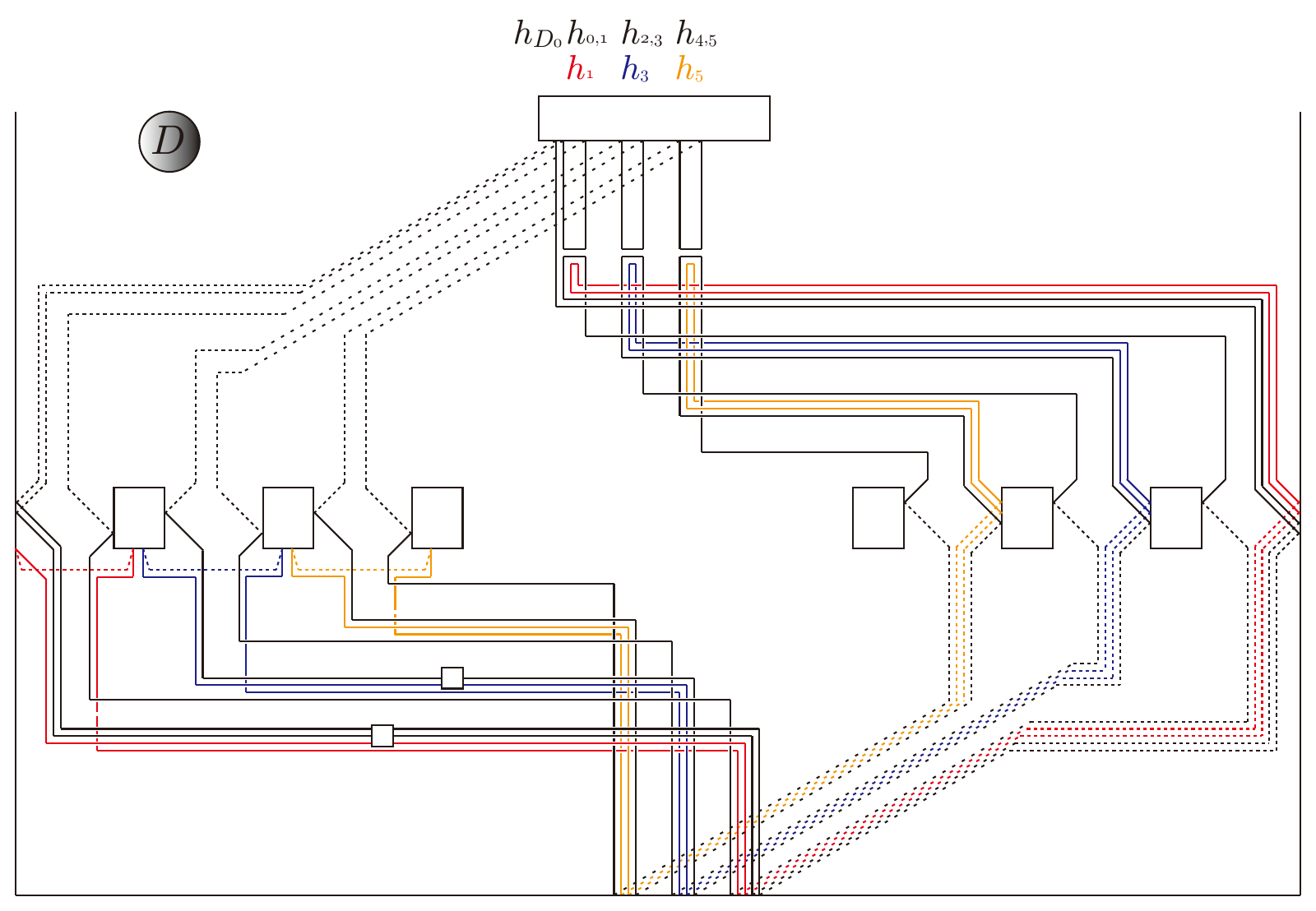}
       \caption{}
       \label{2handles61}
  \end{figure}

\begin{figure}[hbt]
  \centering
       \includegraphics[scale=.70, angle=90]{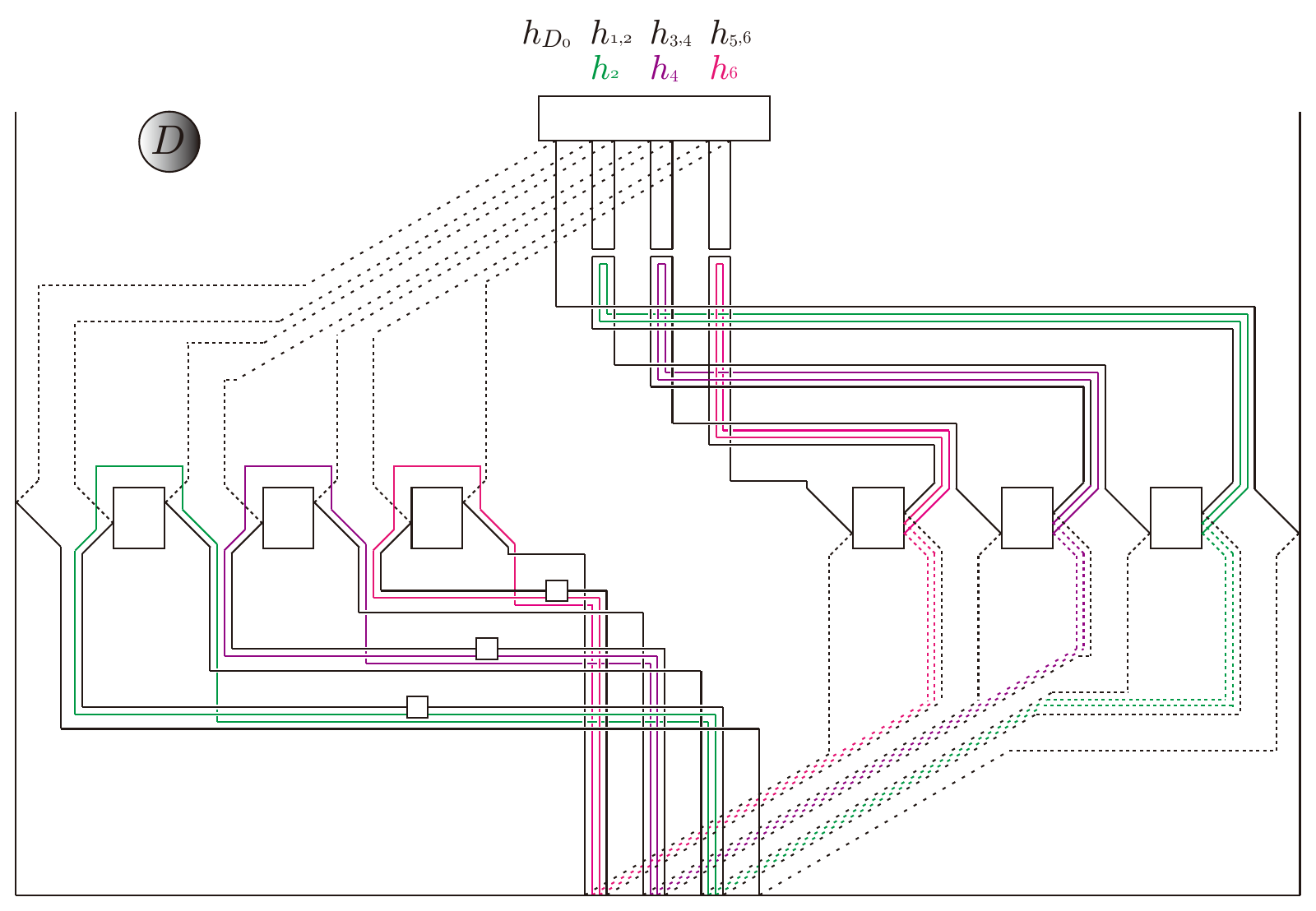}
       \caption{}
       \label{2handles106}
  \end{figure}

\begin{figure}[hbt]
  \centering
       \includegraphics[scale=.70, angle=90]{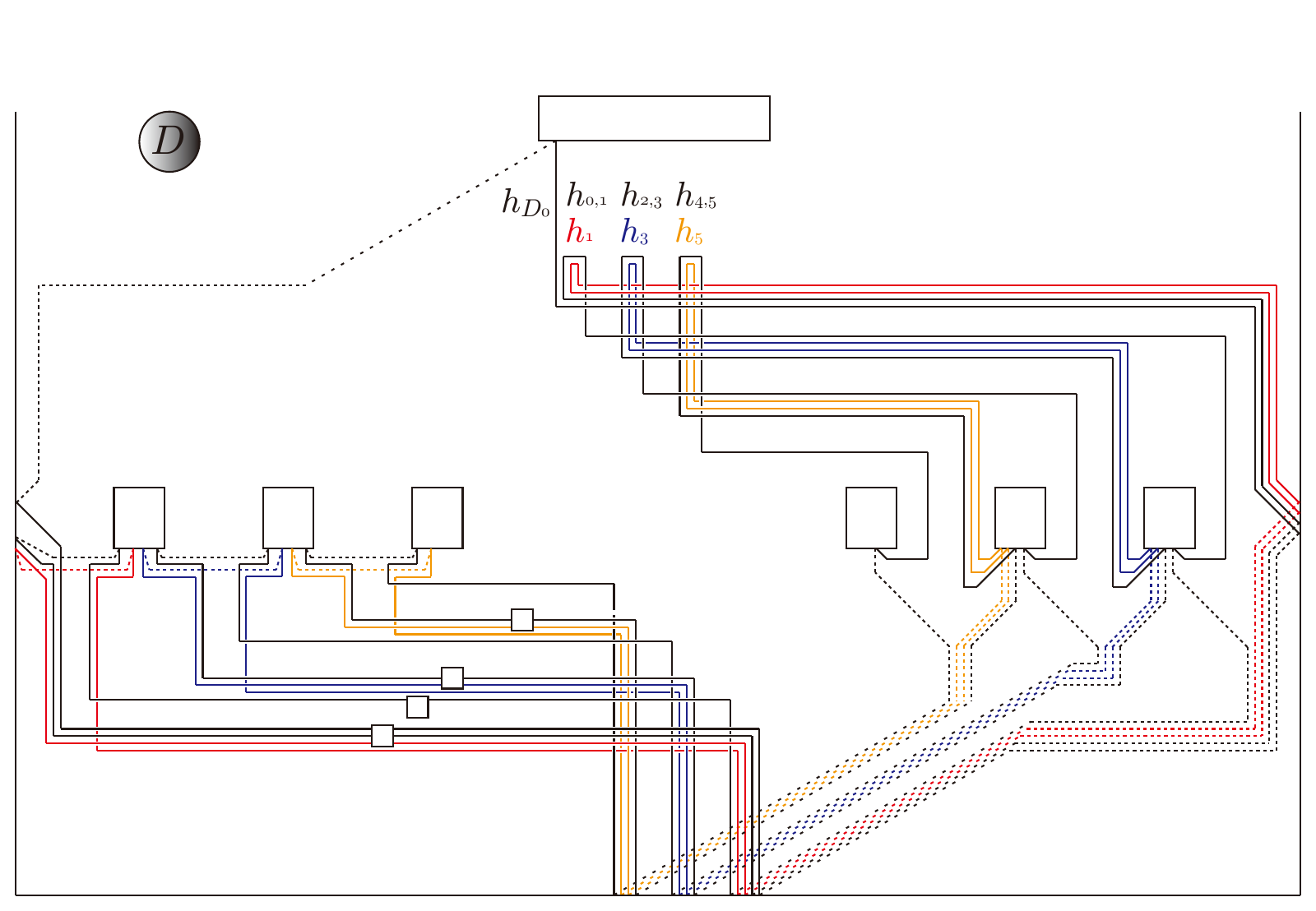}
       \caption{}
       \label{2handles71}
  \end{figure}

\begin{figure}[hbt]
  \centering
       \includegraphics[scale=.70, angle=90]{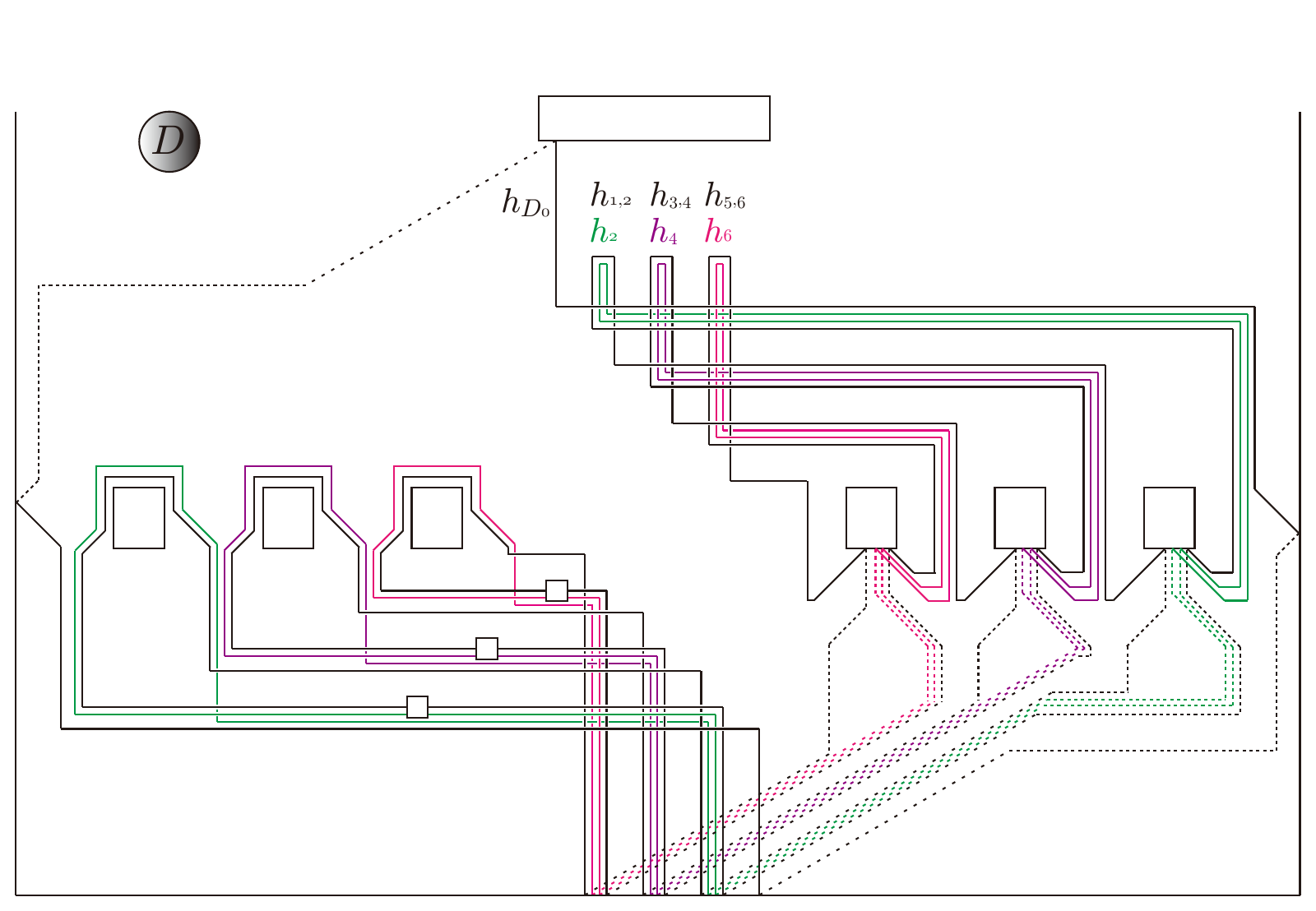}
       \caption{}
       \label{2handles107}
  \end{figure}

\begin{figure}[hbt]
  \centering
       \includegraphics[scale=.70, angle=90]{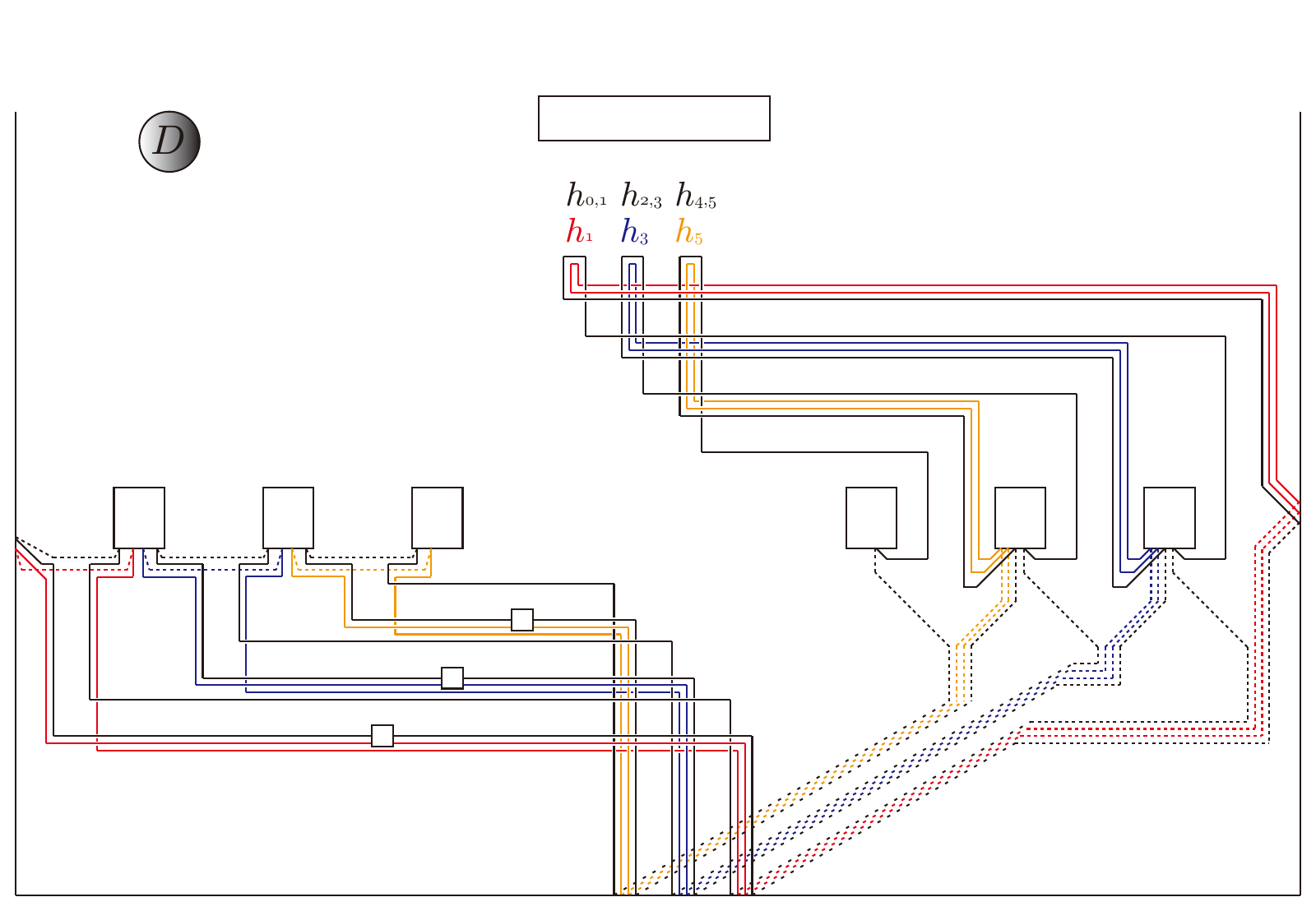}
       \caption{}
       \label{2handles81}
  \end{figure}

\begin{figure}[hbt]
  \centering
       \includegraphics[scale=.70, angle=90]{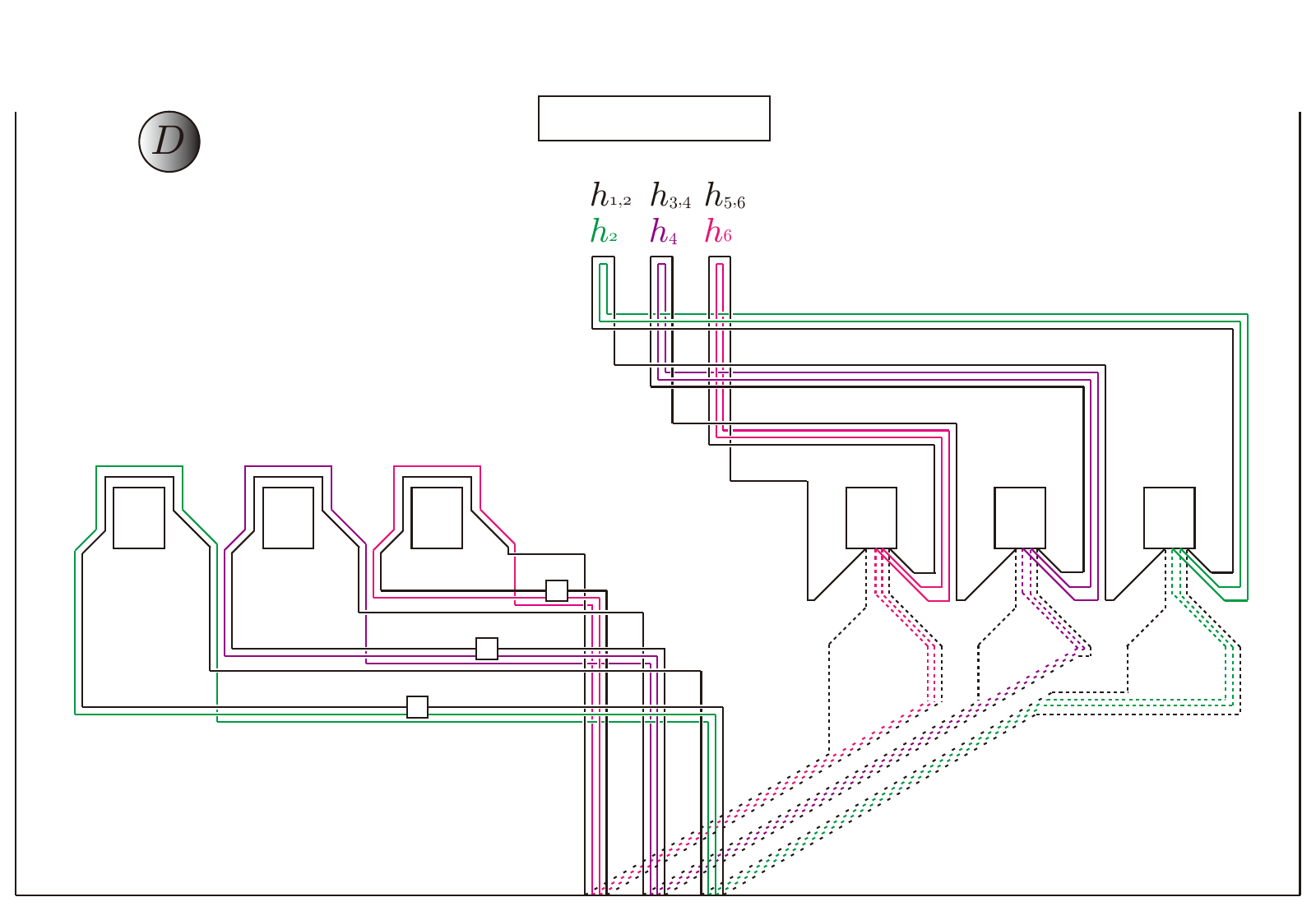}
       \caption{}
       \label{2handles108}
  \end{figure}

\begin{figure}[hbt]
  \centering
       \includegraphics[scale=.70]{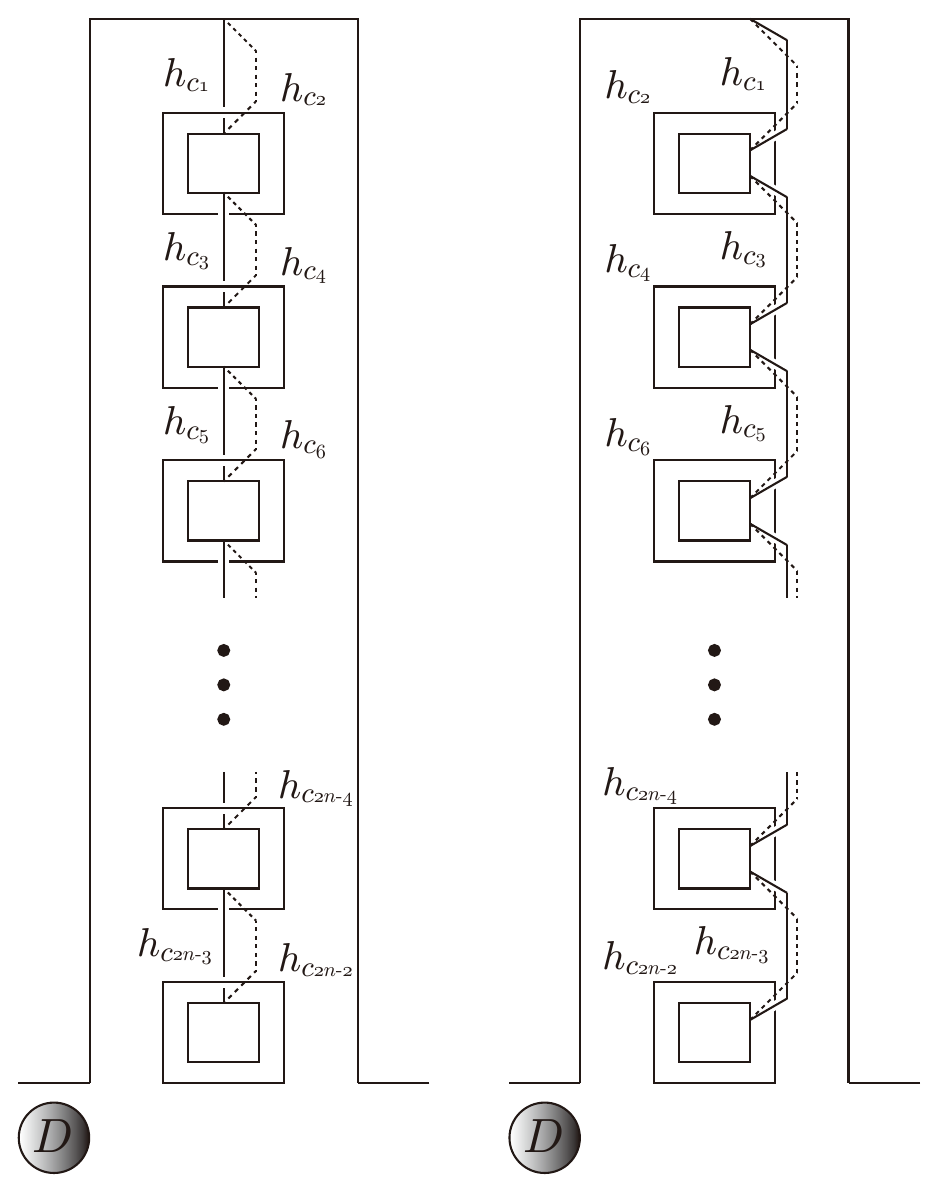}
       \caption{}
       \label{2handles1003}
  \end{figure}

\begin{figure}[hbt]
  \centering
       \includegraphics[scale=.70, angle=90]{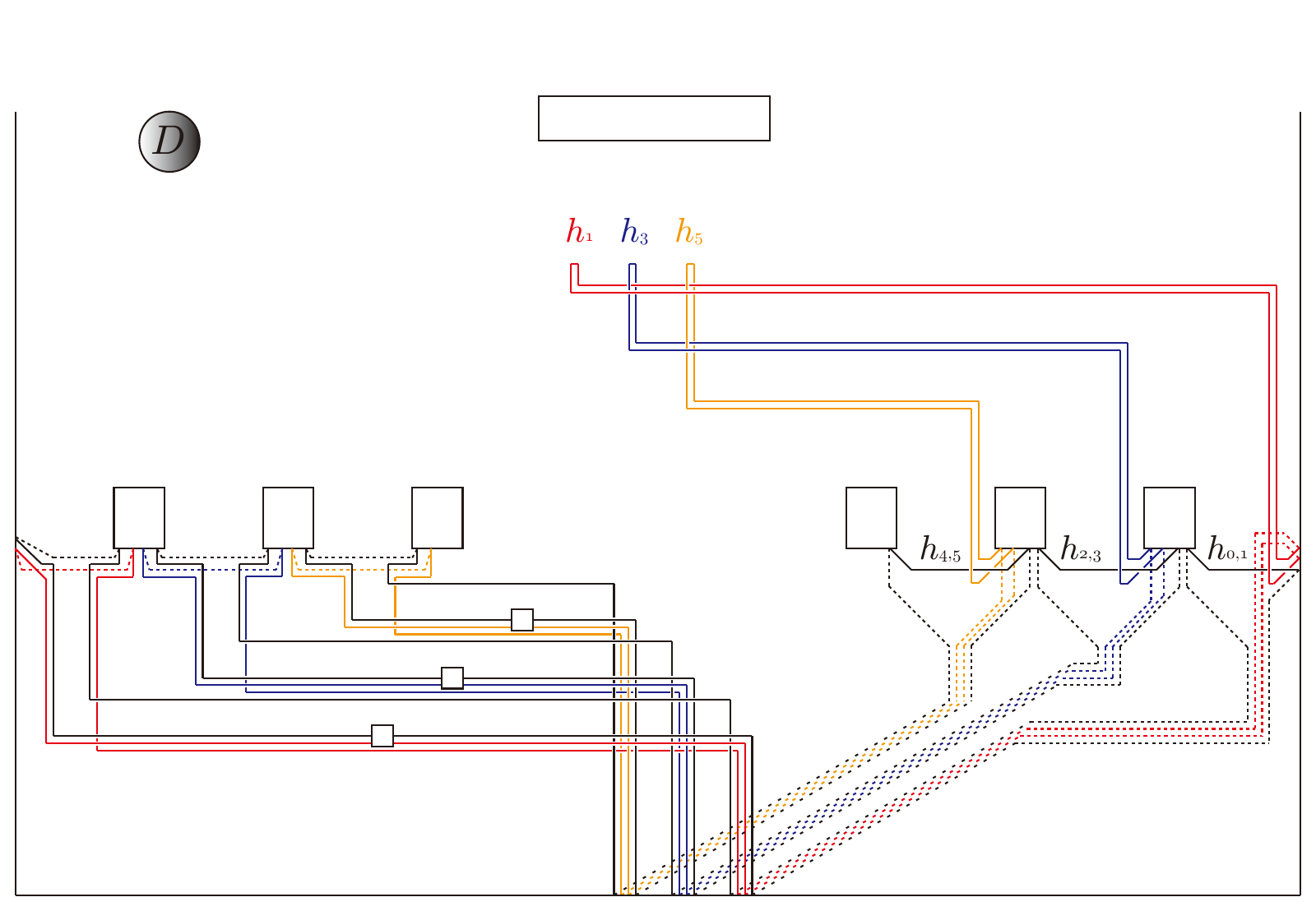}
       \caption{}
       \label{2handles91}
  \end{figure}

\begin{figure}[hbt]
  \centering
       \includegraphics[scale=.70, angle=90]{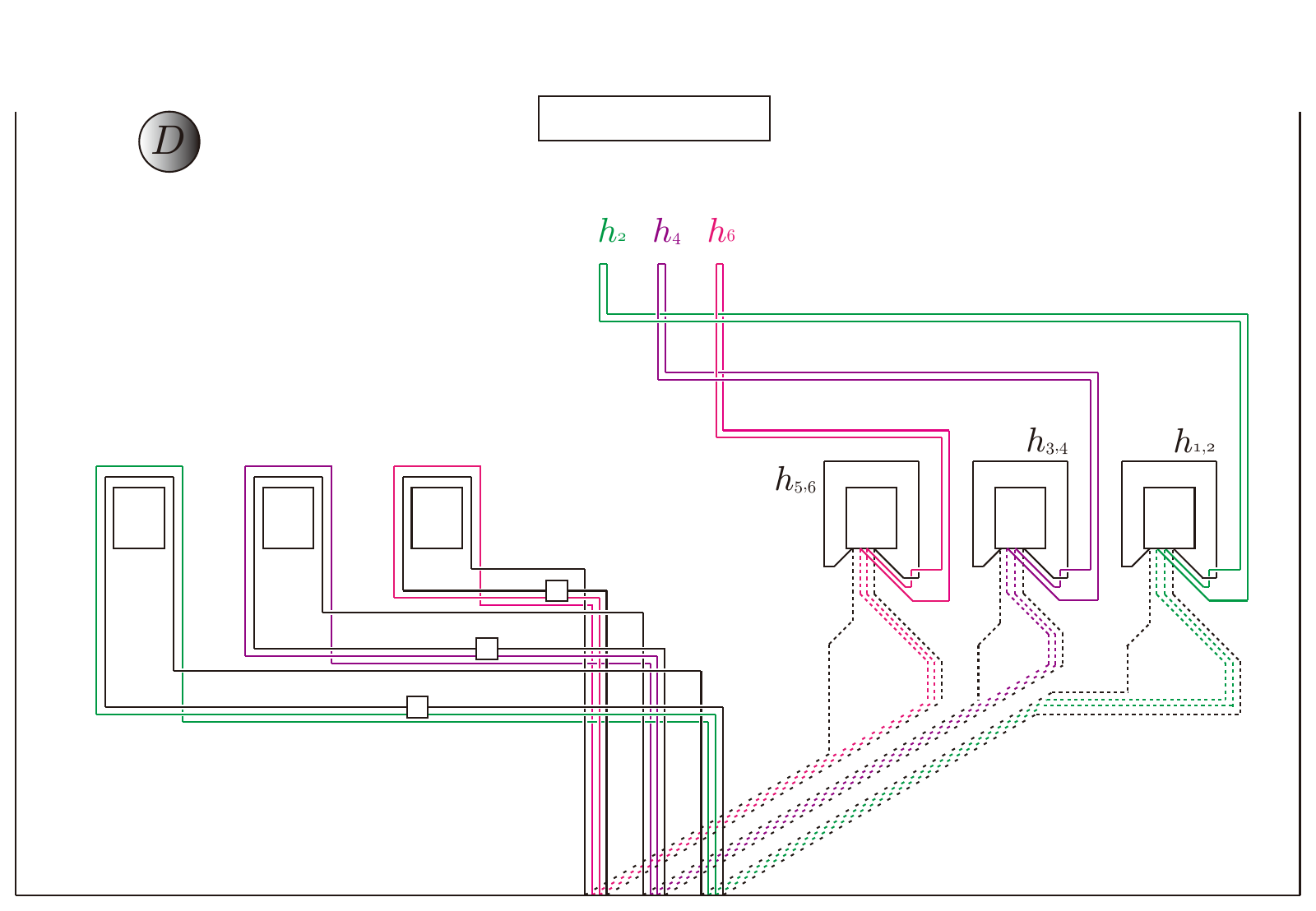}
       \caption{}
       \label{2handles109}
  \end{figure}

\begin{figure}[hbt]
  \centering
       \includegraphics[scale=.70, angle=90]{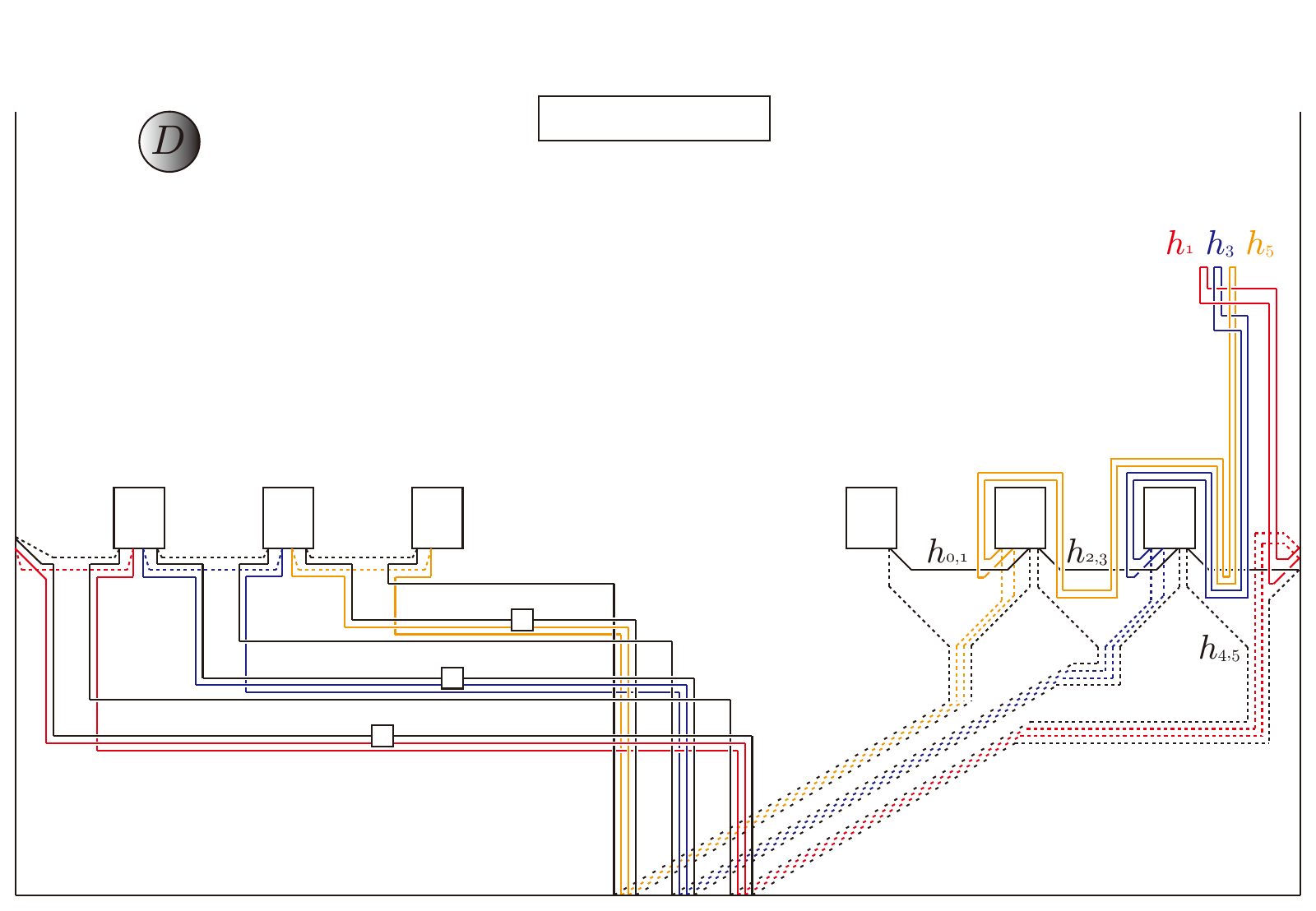}
       \caption{}
       \label{2handles92}
  \end{figure}

\end{document}